\documentclass[11pt]{amsart}

\usepackage{enumerate,url,amssymb,mathrsfs,upref,pdfsync}
\usepackage{amsmath, amsfonts,amsthm,times,graphics,color}
\usepackage{hyperref}
\newtheorem{theorem}{Theorem}[section]
\newtheorem{lemma}[theorem]{Lemma}

\newtheorem{proposition}[theorem]{Proposition}
\newtheorem{conjecture}[theorem]{Conjecture}
\newtheorem*{proposition*}{Proposition}

\theoremstyle{definition}

\newtheorem{corollary}[theorem]{Corollary}

\theoremstyle{remark}
\newtheorem{remark}[theorem]{Remark}

\numberwithin{equation}{section}

\newcommand{\abs}[1]{\lvert#1\rvert}

\newcommand{\A}{\mathbb{A}}

\newcommand{\E}{\mathcal{E}}

\newcommand{\X}{\mathbb{X}}
\newcommand{\Y}{\mathbb{Y}}

\newcommand{\onto}{\overset{{}_{\textnormal{\tiny{onto}}}}{\longrightarrow}}

\DeclareMathOperator{\Mod}{Mod}

\def\XXint#1#2#3{{\setbox0=\hbox{$#1{#2#3}{\int}$}
\vcenter{\hbox{$#2#3$}}\kern-.5\wd0}}

\def\le{\leqslant}
\def\ge{\geqslant}

\begin{document}

\title[Kellogg's theorem for diffeomophic minimisers]{Kellogg's theorem for diffeomophic minimisers of Dirichlet energy between doubly connected Riemann surfaces}  


\author{David Kalaj}
\address{University of Montenegro, Faculty of Natural Sciences and
Mathematics, Cetinjski put b.b. 81000 Podgorica, Montenegro}
\email{davidk@ucg.ac.me}

\begin{abstract}
We extend the celebrated theorem of Kellogg for conformal diffeomorphisms to the minimizers of Dirichlet energy. Namely we prove that a diffeomorphic minimiser of Dirichlet energy of Sobolev mappings between doubly connected Riemanian surfaces  $(\X,\sigma)$ and $(\Y,\rho)$ having $\mathscr{C}^{n,\alpha}$ boundary, $0<\alpha<1$, is $\mathscr{C}^{n,\alpha}$ up to the boundary, provided the metric $\rho$ is smooth enough. Here $n$ is a positive integer. It is crucial that, every diffeomorphic minimizer of Dirichlet energy is a harmonic mapping with a very special Hopf differential and this fact is used in the proof. This improves and extends a recent result by the author and Lamel in \cite{kalam}, where the authors proved a similar result for double-connected domains in the complex plane but for $\alpha'$ which is $\le \alpha$ and $\rho\equiv 1$.  This is  a complementary result of an existence result proved by T. Iwaniec et al. in \cite{iwa} and the author in \cite{kal0}

\end{abstract}

\maketitle
\tableofcontents
\section{Introduction and statement of the main result}\label{intsec}
Throughout this paper $M=(\X,\sigma)$ and $N=(\Y,\rho)$ will be
doubly connected Riemannian surfaces so that $\X$ and $\Y$ are double connected  domains in the complex plane $\mathbf{C}$, where $\rho$ is a
non-vanishing smooth metric defined in $\Y$ so that and $\sigma$ is an arbitrary  metric.

The \emph{Dirichlet energy}
 of a diffeomorphism $f\colon (\X,\sigma) \to (\Y,\rho)$
 is defined by
\begin{equation}\label{ener1}\begin{split}
\E^\rho[f]&= \int_{\X}\|Df(z)\|^2 \, \rho^2(f(z))d \lambda(z)  \\&= 2 \int_{\X}\left(\abs{\partial f(z)}^2 + \abs{\bar \partial f(z)}^2\right) \, \rho^2(f(z))d \lambda(z),\end{split}
\end{equation}
where $\|Df\|$ is the Hilbert-Schmidt norm of the differential matrix of $f$ and $\lambda$ is standard
Lebesgue measure.
The primary goal of this paper is to establish
boundary regularity
of a diffeomorphism $f\colon \X\onto \Y$ of
smallest (finite) Dirichlet energy, provided such
an $f$ exists
and the boundary is smooth. If we denote
by $J(z,f)$ the Jacobian of $f$ at the point $z$,
then~\eqref{ener1} yields
\begin{equation}\label{ener2}
\E^\rho[f] = 2\int_{\X} J(z,f)\, \rho^2(f(z))d \lambda(z) + 4\int_{\X} \abs{\bar \partial f(z)}^2 \rho^2(f(z))d\lambda(z) \ge 2 A(\rho)(\Y),
\end{equation}
where $A(\rho)(\Y)=\int_{\Y}\rho^2(w)d\lambda(w)$. In this paper we will assume that diffeomorphisms as well as Sobolev homeomorphisms are orientation preserving, so that $J(z,f)>0$. A conformal diffeomorphism of $\X$ onto $ \Y$ would be an obvious minimizer of~\eqref{ener2},
because $\bar \partial f=0$, provided it exists. Thus in
the special case where $\X$ and $\Y$ are conformally
equivalent the famous Kellogg's theorem yields that the minimizer is as smooth as the boundary in
the H\"older category. The harmonic mappings come to the stage when the domains are not conformally equivalent. We say a mapping $f:(\X,\sigma)\to (\Y,\rho)$ is harmonic if \begin{equation}\label{hequ}
  \tau(f):=f_{z\bar z}+\frac{\partial \log \rho^2(w)}{\partial w}\circ f(z) \cdot f_z f_{\bar z}\equiv 0.
  \end{equation}
One of important properties of harmonic mappings is the fact that their the so-called Hopf  differential $$\mathrm{Hopf}(f):=\rho^2(f(z))f_z\overline{f}_z$$ is a holomorphic function in $\X$.
For some other important properties of those mappings we refers to the books of J. Jost \cite{Job1, jost, jost3}.
\subsection{Admissible metrics}
Assume that $n\ge 1$ is an integer and $\rho\in \mathscr{C}^{n}$ is a positive function defined in $\Y$.
We call the metric  $\rho$
\emph{admissible} one if it satisfies the following conditions.
 It has a bounded Gauss
curvature $\mathcal{K}$ where
\begin{equation*}\label{gaus}\mathcal{K}(w)=-\frac{\Delta \log \rho(w)}{\rho(w)};\end{equation*}
 It has a finite area defined by
$$\mathcal{A}(\rho)=\int_{\Y}\rho^2(w) du dv, \ \ w=u+iv;$$
   There is a constant $C_\rho>0$ so that \begin{equation}\label{pp}{|\nabla \rho(w)|}\le C_\rho{\rho(w)}, \ \ \ w\in\Y \ \ i.e. \ \ \nabla \log \rho\in L^\infty(\Y)\end{equation} which means that $\rho$ is so-called approximately analytic function (c.f. \cite{EH}). 

Assume that the domain of $\rho$ is the unit disk $\mathbf{D}:=\{z: |z|<1\}\subset\mathbf{C}$.
From \eqref{pp} and boundedness of $\rho$, it follows that it is Lipschitz, and so it is continuous up to the boundary.  Again by using \eqref{pp}, the function $f(t)=\rho(te^{i\alpha})$, $0<t<1$, $\alpha\in[0,2\pi]$ satisfies the differential inequalities $-C_\rho\le \partial_t \log f(t)\le C_\rho$, which by integrating in $[0,t]$ imply that $f(0)e^{-C_\rho t}\le f(t)\le f(0)e^{C_\rho t}$. Therefore under the above
conditions  there holds the double inequality \begin{equation}\label{double}0<\rho(0)e^{-C_\rho} \le {\rho(w)}\le \rho(0)e^{C_\rho}<\infty, \ \ w\in\mathbf{D}.\end{equation}
A similar inequality to \eqref{double} can be proved for $\Y$ instead of $\mathbf{D}$.
  The Euclidean metric ($\rho\equiv 1$) is an admissible metric. The Riemannian metric defined by $\rho(w)={1}/{(1+|w|^2)^2}$ is admissible as well.
The Hyperbolic metric $h(w)={1}/{(1-|w|^2)^2}$ is not an admissible metric on the unit disk neither on the annuli $\A(r,1):= \{z:r<|z|<1\}$,
but it is admissible in $\A(r,R):=\{z: r<|z|<R\}$, where $0<r<R<1$.  In this case the equation \eqref{hequ} leads to hyperbolic harmonic mappings. The class is particularly interesting, due to the recent discovery that every quasisymmetric map of the unit circle onto itself can be extended to a quasiconformal hyperbolic harmonic mapping of the unit disk onto itself. This problem  is known as the Schoen conjecture and it was proved by Markovi\'c in \cite{markovic}.

We  now state the  existence result proved by T. Iwaniec, K.-T. Koh, L. Kovalev, J. Onninen \cite{iwa}  for Euclidean metric and the author  \cite{kal0} for general metrics.

\begin{proposition}\label{q4}
Suppose that $\X$ and $ \Y$ are bounded doubly connected domains in $\mathbf{C}$, where $\X=\A(r,R)$. Assume that $\rho$ is a positive metric with bounded Gaussian curvature
and finite area. Assume that $\mathcal{F}(\X,\Y)$ is the set of a Sobolev homeomorphism between $\X$ and $\Y$. Then

(a) If the solution of the following minimization problem
 \begin{equation}\label{problem}\inf\{\mathscr{E}^\rho[f]: f\in\mathcal{F}(\X,\Y)\}\end{equation} is a diffeomorphism, then it is $\rho-$harmonic, i.e. it satisfies the equation \eqref{hequ}
  and its Hopf differential has the following form \begin{equation}\label{hopf}\mathrm{Hopf}(f)=\frac{\mathbf{c}}{z^2},\end{equation}   where $\mathbf{c}$ is a real constant.

 (b) If $\Mod X\le \Mod \Y$, then there exists a $\rho-$harmonic diffeomorphism that solves the problem \eqref{problem}. In this case $\mathbf{c}\ge 0$.

 \end{proposition}

For an exact statement of the Kellogg's theorem, we
recall that  a function $\xi:D\to \mathbf C$ is
said to be
  uniformly $\alpha-$H\"older continuous and write
$\xi \in \mathscr{C}^{\alpha}(D)$ if
$$\sup_{z\neq w, z,w\in D}\frac{|\xi(z)-\xi(w)|}{|z-w|^\alpha}<\infty.$$
In similar way one
defines  the class  $ \mathscr{C}^{n,\alpha}(D)$
to consist of all functions
$\xi \in \mathscr{C}^n (D)$ which have
their $n-$th  derivative
$D^{(n)} \xi \in \mathscr{C}^\alpha (D)$. A rectifiable Jordan curve $\gamma$ of the length $l=|\gamma|$ is said to be of class $\mathscr{C}^{n,\alpha}$ if its arc-length parameterization $g:[0,l]\to \gamma$ is in $\mathscr{C}^{n,\alpha}$.
The theorem of Kellogg (with an extension
due to Warschawski) states that.
\begin{proposition}[Kellogg and Warschawski;  see
 \cite{G, sw3, w1,w2, chp}]\label{oneone} Let $n\in \mathbb N$,
 $0<\alpha< 1$. If $D$ and $\Omega$ are Jordan domains having
 $\mathscr \mathscr{C}^{n,\alpha}$  boundaries and $\Phi$ is a conformal
 mapping of $D$ onto $\Omega$, then $\Phi^{(n)}\in
 \mathscr{C}^{\alpha}( D)$ and $(\Phi^{-1})^{(n)}\in
 \mathscr{C}^{\alpha}( \Omega)$.
 \end{proposition}

The theorem of
Kellogg and of Warshawski has been extended in various directions, see for example the work on conformal minimal parameterization of  minimal surfaces by Nitsche \cite{nit} (see also the book \cite[Sec.~2.3]{dht} by U. Dierkes, S. Hildebrandt and A. J. Tromba for some extension to the surfaces with prescribed mean curvature as well as the papers \cite{Kid} and \cite{Les0} by Kinderlehrer  and F. D. Lesley respectively ), and to quasiconformal
harmonic
mappings with respect to the hyperbolic metric by Tam and Wan \cite[Theorem 5.5.]{tam}. For some other extensions and quantitative Lipschitz constants we refer to the paper \cite{Lw}.


 We
extend and improve a recent  result obtained by the author and Lamel in \cite{kalam}, where the authors initiated this problem for the Euclidean setting. This  paper also contains a solution of a conjecture posed in that paper. For $\alpha\in(0,1)$ they proved a similar statement for Euclidean metric instead of general $\rho$ and for $\alpha'$ instead of $\alpha$ which is defined by $\alpha'=\alpha$ for $\mathrm{Mod}(\X)\ge \mathrm{Mod}(\Y)$, but $\alpha'=\alpha/(2+\alpha)$ for the case $\mathrm{Mod}(\X)< \mathrm{Mod}(\Y)$ and asked if $\alpha'$ can be replaced by $\alpha$.

We have the following extension of the Kellogg's theorem, which is the main result of this paper.
\begin{theorem} \label{mainexistq}
Suppose  that $M=(\X,\sigma)$ and $N=(\Y,\rho)$ are double  connected Riemannian surfaces, where $\X$ and $\Y$ are  domains in
$\mathbf{C}$ with $\mathscr{C}^{n,\alpha}$  boundaries, $0<\alpha<1$, and let $\rho$ be an admissible metric in $\Y$. Then every $\rho-$ energy minimising  diffeomorphism
between $\X$ and $ \Y$,
has a $\mathscr{C}^{n,\alpha}$ extension  up to the boundary of  $\X$.
\end{theorem}
Theorem~\ref{mainexistq} and Proposition~\ref{q4} imply the following result:
\begin{corollary}\label{pasoja}
Assume that $\X$ and $\Y$ are two doubly connected domains in $\mathbf{C}$ with $\mathscr{C}^{n,\alpha}$ boundaries, $0<\alpha<1$. Assume also that $\Mod(\X)\le \Mod(\Y)$ and that $\rho$ is an admissible metric in $\Y$. Then there exists a diffeomorphic minimizer $h:\X\onto\Y$ of Dirichlet energy $\E^\rho$ and it has a $\mathscr{C}^{n,\alpha}$ extension up to the boundary. Moreover it is unique up to the conformal changes of $\X$ and isometric transformations of $(\Y,\rho)$.
\end{corollary}
\begin{remark}\label{vanesa}

The result is even new for the Euclidean metric at least for the case $\mathrm{Mod}(\X)<\mathrm{Mod}(\Y)$.

The condition \eqref{pp} for the metric $\rho$ can be replaced by the condition  \begin{equation}\label{pppp}|\nabla \log \rho|\in L^p(\Y),\ \  p=2/(1-\alpha),\end{equation} and the proof remains practically unchanged, so  Theorem~\ref{mainexistq} can be generalized a little bit. On the other hand if $\Psi$ is a conformal diffeomorphism of the annulus $\Y$ onto the annulus $\A(R,1)$, then $\rho(w) = |\Psi'(w)|$ is a metric on $\Y$, which by Kellogg's theorem satisfies the condition $|\nabla \log \rho|\in L^{q}(\Y)$, for $q<q_0=1/(1-\alpha)$, provided that the boundary of $\Y$ is $\mathscr{C}^{1,\alpha}$ and the constants $q_0=1/(1-\alpha)$ cannot be improved. Then $\Phi=\Psi^{-1}:\A(r,1)\to (\Y,\rho)$ is a harmonic conformal mappings, (hence a minimizer) which has a $\mathscr{C}^{1,\alpha}$ extension up to the boundary.  This implies that the condition \eqref{pppp} is almost the weakest possible.

Let \begin{equation}\label{nits}f(z)= \frac{r (R-r)}{\left(1-r^2\right) \bar z}+\frac{(1-r R) z}{1-r^2}.\end{equation} Then $f(z)$ is an Euclidean harmonic mapping of the annulus $\A(r,1)$ onto $\A(R,1)$ that minimizes the Euclidean Dirichlet energy (a result proved by Astala, Iwaniec and Martin \cite{AIM}). This result has been extended in \cite{london} for radial metrics.
The mapping is a diffeomorphism  between $\overline{\A}(r,1)$ and $\overline{\A}(R,1)$, provided that
\begin{equation}\label{jjcn}R< \frac{2r}{1+r^2}.\end{equation} If $R= \frac{2r}{1+r^2},$ and $0<r<1$, then the mapping
$$w(z)=\frac{r^2+|z|^2}{\bar z(1+r^2)}$$ is a harmonic minimizer (see \cite{AIM}) of the Euclidean energy of
mappings between $\mathbb{A}(r,1)$ and $\mathbb{A}(\frac{2r}{1+r^2},1)$, however $|w_z|=|w_{\bar z}|=\frac{1}{1+r^2}$ for $|z|=r$, and so $w$ is not
bi-Lipschitz.

\emph{This in turn implies that the inverse of a minimising  diffeomorphism  in Theorem~\ref{mainexistq} is not necessary in $\mathscr{C}^{1,\alpha}(\overline{\Y})$.}

Note that \eqref{jjcn} is satisfied provided that $\Mod \A(r,1)\le \Mod \A(R,1)$. The inequality \eqref{jjcn} (with $\le$ instead of $<$) is necessary and sufficient for the existence of a harmonic diffeomorphism between $\A_r$ and $\A_R$ a conjecture raised by J. C. C. Nitsche in \cite{Nitsche} and proved by Iwaniec,  Kovalev and  Onninen in \cite{IKO2},  after some partial results given by Lyzzaik \cite{L}, Weitsman \cite{W} and the author
\cite{Ka}.

If $R> \frac{2r}{1+r^2},$ then the minimizer of Dirichlet energy throughout the diffeomorphisms between $\A(r,1)$ and $\A(R,1))$ is not a diffeomorphism ( see \cite{AIM} and \cite[Example~1.2]{cris}).
\end{remark}

\begin{remark}
Let $f(z)=\int_0^z \frac{dw}{\sqrt{1-w^4}}$ be a conformal diffeomorphism of the unit disk onto a square.
Then $f$ is a conformal diffeomorphism of the annulus $\A(1/2,1)$ onto the doubly connected,
whose outer boundary is not smooth. We know that $f$ is a minimiser of energy but is not Lipschitz.
 With some more effort, by using e.g. \cite{les} we can define a conformal diffeomorphism between the circular annulus and an annulus with $\mathscr{C}^1$
 boundary so that it is not Lipschitz up to the boundary. This in turn implies that the condition for the annuli to have $\mathscr{C}^{1,\alpha}$ boundary is essential.

   Further an Euclidean harmonic diffeomorphism $f$ of the unit disk $\mathbf{D}$ onto itself is seldom  a
    Lipschitz continuous up to the boundary. We cite here an important result of Pavlovi\'c \cite{MP} which states that harmonic diffeomorphism of the unit disk is
    Lipschitz if it is quasiconformal. Further for such a non-Lipschitz  $f$, let $R<1$.
    Then  the set $\X=f^{-1}(\A(R,1))$ is a doubly-connected domain with $\mathscr{C}^{\infty}$ boundary. Let $\Phi$ be a conformal diffeomorphism of the annulus $\A(r,1)$
     onto $\X$. Then $F=f\circ\Phi$ is a harmonic diffeomorphism between $\A(r,1)$ onto $\A(R,1)$ which is not Lipschitz continuous. This observation tells
      us that there exists a  crucial difference between the harmonic diffeomorphisms between annuli  which are minimizers and those harmonic diffeomorphisms which are not minimizers.
\end{remark}
\subsection{Organization of the paper and outline of the proof}
The paper contains this introduction and six more sections. In the second section we present some results from potential and function theory needed for the proof, where it is also proved the H\"older continuity. The third section contains the proof of the Lipschitz continuity and the forth  sections contains the proof of  $\mathscr{C}^{1,\alpha}$ smooth continuity. The fifth section proves the $\mathscr{C}^{n,\alpha}$ smooth continuity. 

We describe here the idea of the proof. We must emphasis that the idea of the proof which worked for the Euclidean case is not effective in this case. Namely in the case of the Euclidean metric, the harmonic minimizer can be lifted to a certain minimal surface, and this fact has been used by the author and Lamel in \cite{kalam}, in order to prove smoothness of the mapping. In this case a different approach is needed, since $\rho-$harmonic minimizer does not defines a minimal surface in $\mathbf{R}^3$.

It is clear that it is enough to prove that the minimizer $f$ is $\mathscr{C}^{1,\alpha}$ in a neighborhood of an arbitrary boundary point. The problem is reduced to proving that a composition of $f$ with a conformal diffeomorphism $\Phi$ is $\mathscr{C}^{1,\alpha}(D^\circ)$ in a certain domain $D^\circ$, whose boundary contains a Jordan arc that belongs to the boundary of $\X$.   This in turn implies that $f$ itself is $\mathscr{C}^{1,\alpha}(\Phi(D^\circ))$. The first step is to prove that $f$ is H\"older continuous for a certain constant $\beta<1$. This is proved by using the fact invented in \cite{kal} that the diffeomorphic minimizers are $(K,K')-$quasiconformal. Further we improve this H\"older continuous constant, by using a Korn-Privalov type result due to J.C. C. Nitsche (\cite{nit}) successively for $\beta_j=(1+\alpha)^j\beta$, $j=1,\dots,k$ until we eventually reach $\beta_{k+1}>1$, by using the given $\mathscr{C}^{1,\alpha}$ smoothness of the boundary curve and special Hopf differential of the mapping $f$. Then we obtain that $f$ is Lipschitz continuous on $\X$. In order to get smoothness of the mapping up to the boundary, we previously prove that $f=u+iv\in \mathscr{C}^{1,\alpha/2}$. This is done by using the key lemma proved in Section~\ref{sec4}. This lemma asserts that a function $Y=u-\gamma(v)$ is $\mathscr{C}^{1,\alpha}$ in a neighborhood of a boundary point, where $(\gamma(y),y)$, $y\in(-\epsilon,\epsilon)$ is the graphic of a  is a certain boundary portion.  Further by writing the Hopf differential in the form  $\mathrm{Hopf}(f)(z)=u_z^2+v_z^2=A$, for some smooth function $A$, enables us to conclude that a certain boundary function defined  in \eqref{rree} is real and $\mathscr{C}^{\alpha}$ continuous. This is a crucial point, where we obtain that a given function has a continuous square root, which is $\mathscr{C}^{\alpha/2}$ continuous. By using some well-known estimates of the Green potential and the particular form of Hopf differential of $f$, we aim to get that $f\in\mathscr{C}^{1,\alpha/2}$. Further by using one more time the Korn-Privalov type result we get that the function $f$ is $\mathscr{C}^{1,\alpha}$. 
The case $n\ge 2$ is much easier and we use the mathematical induction and Proposition~\ref{2222}.
At the end of the paper we present
an attractive conjecture.

\section{Auxiliary results}

\subsection{$(K,K')-$quasiconformal mappings }\label{stasec}

A sense preserving mapping $w$ of class ACL between two planar domains $\X$ and $\Y$ is called $(K, K')$-quasi-conformal if \begin{equation}\label{map}\|Dw\|^2\le 2KJ(z,w)+K',\end{equation} for almost every $z\in \X$. Here $K\ge 1, K'\ge 0$, $J(z,w)$ is the Jacobian of $w$ in $z$ and $\|Dw\|^2=|w_x|^2+|w_y^2|=2|w_z|^2+2|w_{\bar z}|^2$.


Mappings which satisfy Eq. \eqref{map} arise naturally in elliptic equations, where $w =
u + iv$, and $u$ and $v$ are partial derivatives of solutions (see  \cite[Chapter~XII]{gt} and the paper of Simon \cite{simon}).

\begin{lemma}\label{popi}\cite{kal}
Every diffeomorphic minimizer of $\rho-$Dirichlet energy between doubly-connected domains $\A(r,1)$ and $\Y$ is $(K,K')$ quasiconformal, where $$K=1 \ \ \text{and}\ \ K'=\frac{2\abs{\mathbf{c}}}{r^2\inf_{w\in \Y}\rho(w)},$$ and $\mathbf{c}$ is the constant from \eqref{hopf}. The result is sharp and for $\mathbf{c}=0$ the minimizer is $(1,0)$ quasiconformal, i.e. it is a conformal mapping. In this case $\Y$ is conformally equivalent with $\A(r,1)$.
\end{lemma}
\subsection{H\"older property of minimizers}\label{holderi}
We first formulate the following result
\begin{proposition}[Caratheodory's theorem  for $(K,K')$ mappings]\cite{kalmat}\label{cara} Let $W$ be
a simply connected domain in $\overline{\mathbf{C}}$ whose boundary
has at least two boundary points such that $\infty\notin \partial
W$. Let $f : \mathbf{D} \rightarrow W$ be a  continuous mapping of
the unit disk $\mathbf{D}$ onto $W$ and $(K,K')$ quasiconformal near
the boundary $\mathbf T$.

Then $f$ has a continuous extension up to
the boundary if and only if  $\partial W$ is locally connected.\\
\end{proposition}

 Let $\Gamma$ be a rectifiable
Jordan curve and let $g$ be the arc length parameterization of
$\Gamma$ and let $l=|\Gamma|$ be the length of $\Gamma$.  Let
$d_\Gamma$ be the distance between $g(s)$ and $g(t)$ along the curve
$\Gamma$, i.e.
\begin{equation}\label{kernelar3}d_\Gamma(g(s),g(t))=\min\{|s-t|,
(l-|s-t|)\}.\end{equation}

A closed rectifiable Jordan curve $\Gamma$ enjoys a $b-$ chord-arc
condition for some constant $b> 1$ if for all $z_1,z_2\in \Gamma$
there holds the inequality
\begin{equation}\label{24march}
d_\Gamma(z_1,z_2)\le b|z_1-z_2|.
\end{equation}
It is clear that if $\Gamma\in \mathscr{C}^{1}$ then $\Gamma$ enjoys a
chord-arc condition for some $b=b_\Gamma>1$. In the following lemma we use the notation $\Omega(\Gamma)$ for a Jordan domain bounded by the Jordan curve $\Gamma$. Similarly, $\Y(\Gamma, \Gamma_1)$ denotes the doubly connected domain between two Jordan curves $\Gamma$ and $\Gamma_1$, such that $\Gamma_1\subset \Omega(\Gamma)$.

 In this section we prove  that the minimizers of the energy are global H\"older continuous  provided that the boundary is  $\mathscr{C}^{1}$.
\begin{lemma}\label{newle}\cite{kalam}
Assume that the Jordan curves $\Gamma,\Gamma_1$ are in the class $\mathscr{C}^1$.
Then there is a constant $B>1$, so that  $\Gamma$ and $\Gamma_1$ satisfy $B-$ chord-arc condition and for every $(K,K')-$ q.c.  mapping $f$ between the
annulus $\X = \A(r,1)$ and the doubly connected domain $\Y=\Y(\Gamma,\Gamma_1)$, bounded by $\Gamma$ and $\Gamma_1$, there exists a positive constant $L=L(K,K',B, r, f)$ so that  there holds \begin{equation}\label{enjte}|f(z_1)-f(z_2)|\le
L|z_1-z_2|^\beta\end{equation} for $z_1,z_2\in \mathbf T$ and $z_1,z_2\in r\mathbf{T}$ for $\beta
= \frac{1}{K(1+2B)^2}.$
\end{lemma}
In view of  Proposition~\ref{cara}, Lemma~\ref{popi}, Lemma~\ref{newle}, Lemma~\ref{heliu} and local representation \eqref{lokrepre} we can formulate the following simple proposition
\begin{proposition}\label{newcara}
Assume that $f$ is a diffeomorphic minimiser of Dirichlet energy between the annuli $\X=\A(r,1)$ and $\Y$, where $\Y$ is doubly connected bounded by the outer boundary $\Gamma$ and inner boundary $\Gamma_1$. Then $f$ has a $\beta-$H\"older continuous extension up to the boundary. 
\end{proposition}

 \subsection{Some auxiliary results from potential and function theory}\label{potent}

We first formulate two propositions needed in the sequel.
\begin{proposition}\label{trudi}\cite[Corollary~8.36]{gt}.
Let $T$ be a $\mathscr{C}^{1,\lambda}$ portion of the boundary of a Jordan domain $\Omega\subset \mathbf{C}$ and assume that $\omega\in\mathscr{W}^{1,2}(\Omega)$ is a weak solution of $\Delta \omega = g$, where $g\in L^\infty(\Omega)$, or more general $g\in L^{2/(1-\lambda)}(\Omega)$. Assume further that $\omega|_{T}\equiv 0$. Then $\omega\in \mathscr{C}^{1,\lambda}(\Omega\cup T)$, and  for every relatively compact subset $\Omega'$ of $\,\Omega\cup T$, there is a constant $C=C(n,\mathrm{dist}(\Omega',\partial\Omega\setminus T),T)$ so that $$\|\omega\|_{\mathscr{C}^{1,\lambda}}\le C(\|g\|_{L^\infty}+\|\omega\|_{L^\infty}),
\ \ \text{and}\ \  \ \|\omega\|_{\mathscr{C}^{1,\lambda}}\le C(\|g\|_{L^p}+\|\omega\|_{L^\infty}),$$ $p=2/(1-\lambda)$.
\end{proposition}
\begin{proposition}\cite[Theorem~6.19]{gt}\label{2222}
Let $k$ be a non-negative integer and let $\Omega$ be a Jordan domain in $\mathbf{C}$ with a boundary portion $T\subset\partial\Omega$ so that $T\in \mathscr{C}^{k+2,\alpha}$. Let $U\in W^{2,p}(\Omega\cup T)$ be a strong solution of $\Delta  U=Q$, where $Q\in \mathscr{C}^{k,\alpha}$ with  $Q|_{T}\equiv 0$. Then $U\in \mathscr{C}^{k+2,\alpha}(\Omega\cup T)$.
\end{proposition}
Proposition~\ref{2222} looks more general than \cite[Theorem~6.19]{gt}, however its proof in \cite{gt}, as the authors of \cite{gt} pointed out applies also to this version.


By repeating the proof of the theorem of Hardy and Littlewood, \cite[Theorem~3, p.\ 411]{G} and \cite[Theorem~4, p.\ 414]{G}, we can state the following two lemmas.

\begin{lemma}\label{hali} Let $\mu\in(0,1)$ and let $0<r<1$ and $s_0\in[0,2\pi)$.
Assume that $f$ is a holomorphic mapping defined in the unit disk so that $$|f'(z)|\le M(1-|z|)^{\mu - 1},$$ where $0<\mu<1$ and $z\in\{re^{i(s+s_0)}: 1/2\le r\le 1, s\in(-r,r)\}$. Then the radial limit $$\lim_{\tau\to 1-0}f(\tau e^{i\theta})=f(e^{i\theta})$$ exists for every $\theta\in (-r+s_0,r+s_0)$ and we have there the inequality
$$|f(w)-f(w')|\le N |w-w'|^\mu, \ \ w,w'\in \{re^{i(s+s_0)}: 1/2\le r\le 1, s\in(-r,r)\},$$ where $N$ depends on $M$ and $\mu$.  The converse is also true.
\end{lemma}

Now we formulate some required facts from the function theory \cite[Chapter~IX]{G} and \cite[Lemma~7]{nit}.

\begin{lemma}\label{heliu} Let $\mu\in(0,1)$.
Assume that $f$ is continuous harmonic  mapping on the closed unit disk and satisfies on a small arc $\Lambda=\{e^{i\theta}: |\theta-s_0|<r\}$ the  condition: $$|f(e^{is})-f(e^{it})|\le A|t-s|^\mu, \ \ e^{it}, e^{is}\in\Lambda,$$  for  almost every point $s$ and $t$. Then $f$ satisfies the H\" older condition $$|f(z) - f(w)|\le B |z-w|^\mu$$ for $z,w\in \{r e^{is}: 1-r\le r\le 1, s\in (-r+s_0,r+s_0)\}$.
\end{lemma}

In order to continue  we recall a  Korn-Privalov type results by J. C. C. Nitsche (\cite[Lemma~7]{nit} and a relation from its proof).
\begin{lemma}\label{loclema}
Assume that $F$  is a bounded holomorphic mapping defined in the unit disk, so that $| F|\le M$ in $\mathbf{D}$. Further assume that for a constants  $0< \ell$, $0< \eta,\mu\le \pi/2$ so that for almost every $-\ell\le t,s\le \ell$ we have
$$|\Re F(t)-\Re F(s)|\le M |t-s|^\mu \{\min\{|t|^\eta,|s|^\eta\}+|t-s|^\eta\}.$$ Then for $\zeta=\tau e^{is}$, with $|s|\le \ell/2$, $1/2 \le \tau\le 1$  we have the estimates

\begin{equation}\label{estimate}|F'(\zeta)|\le \left\{
                        \begin{array}{ll}
                          M_1|s|^\eta(1-\tau)^{\mu -1}+M_2(1-\tau)^{\mu +\eta-1}+M_3, & \hbox{if $\mu +\eta<1$;} \\
                          M_1|s|^\eta(1-\tau)^{\mu -1}+M_2\log \frac{1}{1-\tau}+M_3, & \hbox{if $\mu +\eta=1$;} \\
                          M_1|s|^\eta(1-\tau)^{\mu -1}+M_2, & \hbox{if $\mu<1 \wedge \mu +\eta>1$;} \\
                          M_1|s|^\eta \cdot \log \frac{1}{1-\tau} +M_3, & \hbox{if $\mu=1$;} \\
                          M_1, & \hbox{if $\mu>1$;}
                        \end{array}
                      \right.
\end{equation}
and
\begin{equation}\label{rho}
|F(\tau)-F(1)|\le \left\{
                    \begin{array}{ll}
                      N(1-\tau)^{\mu +\eta}, & \hbox{ if $\mu+\eta<1$;} \\
                        N(1-\tau)\log \frac{1}{1-\tau}, & \hbox{ if $\mu+\eta=1$;} \\
                        N(1-\tau), & \hbox{ if $\mu+\eta>1$,}
                    \end{array}
                  \right.
\end{equation}

Here $N$, $M_1,M_2,M_3$ depends on $M, \eta, \mu$ and $\ell$.
\end{lemma}

\section{Proof of Lipschitz continuity}\label{sectio}
In the sequel we prove Lipschitz continuity for all the range $\alpha\in(0,1)$.
Assume that $f:\X\onto\Y$ is a diffeomorphic minimizer. Then there is $r<1$ and a conformal difeomorphism $\Psi$ of the annulus $\A(r,1)$ and $\X$. Then the mapping $f\circ \Psi$ is a diffeomorphic minimizer and also $\rho-$harmonic. Moreover $\Psi$ has $\mathscr{C}^{1,\alpha}$ extension to the boundary if and only if $\partial\X\in  \mathscr{C}^{1,\alpha}$. This is why in the sequel we assume that $\X=\A(r,1)$.
Assume that $a\in\partial \X$ and $b=f(a)\in \partial\Y$. Assume that $n_b$ is a unit tangent vector of $\partial \Y$ at $b$.  Let $\X_a$ be a $\mathscr{C}^{1,\alpha}$ Jordan domain, symmetric w.r.t. the ray $\mathrm{arg}\, z = \mathrm{arg}\, a$ so that $\partial\X_a\cap \partial \X$ is the Jordan arc $a e^{it},t\in [-1,1]$. Assume that $\Phi=\Phi_a:\mathbf{D}\to \X_a$ is a conformal diffeomorphism so that $\Phi_a(1) = a$ and $\epsilon$ is such a constant so that the arc $T_\epsilon:=\{e^{it}, t\in [-\epsilon,\epsilon]\}$ is mapped by $\Phi_a$ onto $\{a e^{it}:t\in [-1,1]\}$ for every $a$.
Let \begin{equation}\label{faxf}F(z)=F^a(z) =\overline{n_b}\left(f(\Phi_a(z))-f(a)\right).\end{equation} Then $F^a$ is a diffeomorphism of the unit disk $\mathbf{D}$ onto  $\Y_a=\overline{n_b}\left(f(\X_a)-f(a)\right)$.  For $\varepsilon>0$ define $D_\varepsilon = \{z=re^{it}\in\mathbf{D}: 1-\varepsilon\le r<1 \wedge t\in (-\varepsilon,\varepsilon)\}$. Then $\partial D_\varepsilon \cap \mathbf{T}=T_\varepsilon$.

Let $v(t) = \Im F^a(e^{it})$. Let $\ell_a$ be the length of an arc $\Gamma_a\subset \partial (\overline{n_b}\left(f(\X)-f(a)\right))$ containing $0$ in its "center" so that $\partial\Y_a\cap\partial (\overline{n_b}\left(f(\X)-f(a)\right))\subset\Gamma_a$ and assume that $$\Gamma:[-\ell_a/2,\ell_a/2]\to \Gamma_a$$ is a length-arc parameterization so that $\Gamma(0) = 0$. Then $\Gamma'(0) = \overline{n_b} n_b=|n_b|^2=1$.  Let $y(s) = \Im \Gamma(s)$. Then $y(0) = 0$, $y'(0) = \Im \Gamma'(0)=0$. Further $\Gamma\in \mathscr{C}^{1,\alpha}([-\ell_a/2,\ell_a/2])$  and thus $y\in \mathscr{C}^{1,\alpha}([-\ell_a/2,\ell_a/2])$.

Therefore for $s_1,s_2\in [-\ell_a/2,\ell_a/2]$, there is $\tau\in [\min\{s_1,s_2\},\max\{s_1,s_2\}]$, and a constant $C=C(\alpha,\partial\Y)$ so that  $$|y(s_1)-y(s_2)|=|s_1-s_2| |y'(\tau)|=|s_1-s_2| |y'(\tau)-y'(0)|\le C|s_1-s_2| |\tau|^\alpha.$$
In the course a proof the value of a constant $C$ may change from one occurrence to the next.
Now since $$|\tau|^\alpha\le \max\{|s_1|^\alpha,|s_2|^\alpha\}$$ and $$\max\{|s_1|^\alpha,|s_2|^\alpha\}\le \min\{|s_1|^\alpha,|s_2|^\alpha\}+|s_1-s_2|^\alpha$$ for $\alpha\in (0,1)$, we get

\begin{equation}\label{yyy1}|y(s_1)-y(s_2)|\le C|s_1-s_2|\left(\min\{|s_1|^\alpha,|s_2|^\alpha\}+|s_1-s_2|^\alpha\right),\end{equation} for $s_1,s_2\in [-\ell_a/2,\ell_a/2]$.

Further let $\phi:[-\epsilon ,\epsilon]\to [-\ell_a/2,\ell_a/2]$ be the function defined by  $\phi(t) = \Gamma^{-1}(F(e^{it}))$. Then $v(0)=0$ and \begin{equation}\label{vphi}v(t) = \Im \Gamma(\phi(t))=y(\phi(t)).\end{equation} Since $$\Gamma:[-\ell_a/2,\ell_a/2)\onto \Gamma_a$$ is a $\mathscr{C}^{1,\alpha}$ diffeomorphism, it follows that $\phi$ and $F|_{T_\epsilon}$ have the same regularity. In view of Lemma~\ref{newle}, $F|_{T_\epsilon}$ is $\beta-$H\"older continuous and so $\phi$.
Thus \begin{equation}\label{phib}|\phi(t_1) - \phi(t_2)|\le L_1 |t_1-t_2|^\beta.\end{equation}
By combining \eqref{yyy1}, \eqref{vphi} and \eqref{phib} we get
\begin{equation}\label{yyy}|v(t_1)-v(t_2)|\le C L_1^{1+\alpha}|t_1-t_2|^\beta \left(\min\{|t_1|^{\alpha \beta},|t_2|^{\alpha \beta}\}+|t_1-t_2|^{\alpha\beta}\right),\end{equation} for $t_1,t_2\in [-\epsilon,\epsilon]$.

Observe that $F$ is a solution of $\rho_a-$harmonic equation, where
$\rho_a$ is a metric in $\overline{n_b}(f(\X_a)-b)$ defined by $\rho_a(w) = \rho({n_b}w+b)$.
Namely \begin{equation}\label{spsp}\begin{split}\tau (F(z))&=F_{z\bar z}+\frac{\partial \log \rho^2_a(w)}{\partial w}\circ F \cdot F_z F_{\bar z}\\&=\overline{n_b}|\Phi_a'(z)|^2 \left(f_{z\bar z}+\frac{\partial\log \rho^2(w)}{\partial w}\circ f \cdot f_z f_{\bar z}\right)\equiv 0. \end{split}\end{equation}
Moreover from  \eqref{hopf} \begin{equation}\label{hopfan}
F_z\overline{F}_z = \frac{\mathbf{c}(\Phi_a'(z))^2}{\rho^2_a(F(z))\Phi_a^2(z)}.
\end{equation}
Thus by \eqref{spsp} we get \begin{equation}\label{hgre}
F_{z\bar z}=-\overline{n_b}|\Phi_a'(z)|^2\frac{\partial\log \rho^2(w)}{\partial w}\circ f(\Phi_a(z)) \cdot f_z(\Phi_a(z)) f_{\bar z}(\Phi_a(z)).
\end{equation}
Since $$|f_z(\Phi_a(z)) f_{\bar z}(\Phi_a(z))|=|f_z(\Phi_a(z)) \bar f_{z}(\Phi_a(z))|$$ we get from \eqref{hequ} and \eqref{hopfan}
the estimate \begin{equation}\label{hgre1}
|F_{z\bar z}|\le |\Phi_a'(z)|^2\left|\frac{\partial \log \rho^2(w)}{\partial w}\circ f\right| \frac{|\mathbf{c}|}{|\Phi_a(z)|^2 }\le C.
\end{equation}
Define the constant
\begin{equation}\label{phi0}\Phi_0 = \max_{a\in \partial \X}\max_{z\in D_\epsilon}\frac{\sqrt{|\mathbf{c}|}|\Phi_a'(z)|}{\rho_a(F(z))|\Phi_a(z)|}.\end{equation} It is clear that $\Phi_0$ exists and is finite.

Let $f_\circ:\mathbf{T}\to \partial \Y_a$ be the mapping defined by $$f_\circ(e^{it}) =F(e^{it}).$$
Then we have
\begin{equation}\label{lokrepre}F=P[f_\circ]+G[\Delta f],\end{equation} where 
$$P[\xi](re^{is}) = \frac{1}{2\pi}\int_0^{2\pi}\frac{1-r^2}{1+r^2-2r \cos(t-s)}\xi(e^{it})dt$$ is the Poisson integral of $\xi:\mathbf{T}\to\mathbf{C}$ and 
$$G[h](z) =\frac{1}{\pi}\int_{\mathbf{D}}\log\frac{|w-z|}{|1- w\overline{z}|} h(w) d\lambda(w)$$ is the Green potential of $h:\mathbf{D}\to \mathbf{C}$.

Let \begin{equation}\label{fff}H(z) = g(z)+\overline{h(z)}=P[f_\circ](z)  \text{ and  }\omega(z) = G[\Delta F](z).\end{equation}

Since $\Delta F=4F_{z\bar z}$ is bounded, by Proposition~\ref{trudi}, $\omega$ has a $\mathscr{C}^{1,\alpha}$ extension in $\mathbf{T}=\partial \mathbf{D}$. Moreover \begin{equation}\label{bomega}\|D\omega(z)-D\omega(z')\|\le C|z-z'|^\alpha, \ z,z'\in \mathbf{D}.\end{equation}
Observe that for $t\in [-\epsilon,\epsilon]$ we have $$v(t) = \Im H(e^{it})=\Im (f_\circ(e^{it})-\omega(e^{it}))=\Im (g(e^{it})+\overline{h(e^{it})})=\Im(g(e^{it})-h(e^{it})).$$

So \begin{equation}\label{vrev}  v(t)=\Re\left(i\left(h(e^{it})-g(e^{it})\right)\right).
\end{equation}

Now we choose such $\beta$, by diminishing the $\beta$ from Proposition~\ref{newcara} if needed so that for some positive integer $k$, $(1+\alpha)^k\beta <1<(1+\alpha)^{k+1}\beta$. Note that $\beta<1/2$ and $\alpha<1$ and so $k\ge 1$.

Then from \eqref{estimate}, \eqref{yyy} and \eqref{vrev} we get \begin{equation}\label{mima}
|i(h'(z)-g'(z))|(1-|z|)^{1-(1+\alpha)\beta}\le M_2,\ \ \ z\in(1-\epsilon,1].
\end{equation}
Remember that $M_2$ does depend exclusively on $\partial\Y$ and not on specific value of $z$.
Further observe that
\begin{equation}\label{subtle}\left(i (F_z-\overline{F}_{ z})\right)^2+\left(F_z+\overline{F}_{ z}\right)^2= \frac{4\mathbf{c}(\Phi'(z))^2}{\rho^2_a(F(z))\Phi^2(z)}.\end{equation}
Furthermore
\begin{equation}\label{fzome}F_z=g'+\omega_z, \ \  \overline{F}_{ z} = h'+\overline{\omega}_{ z},\end{equation} and so from
\eqref{mima} and \eqref{bomega} we get \begin{equation}\label{fzfz}\left|i (F_z(z)-\overline{F}_{ z}(z))\right|(1-|z|)^{1-(1+\alpha)\beta}\le C_2=M_2+2\|D\omega\|_\infty, \ \ z\in(1-\epsilon,1].\end{equation}
Then from \eqref{subtle} and \eqref{fzfz} we get
\begin{equation}\label{secondstep}
\left|F_z(z)+\overline{F}_{ z}(z)\right|(1-|z|)^{1-(1+\alpha)\beta}\le C_3=\sqrt{\Phi_0^2+C_2^2},\ \ z\in [1-\epsilon,1]
\end{equation}
where $\Phi_0$ is defined in \eqref{phi0}.

By combining \eqref{fzfz} and \eqref{secondstep}, we get
 \begin{equation}\label{mima2}
(|F_z(z)|+|F_{\bar z}(z)|)(1-|z|)^{1-(1+\alpha)\beta}\le C_4=C_2+C_3,\ \ \ z\in [1-\epsilon,1].
\end{equation}
Here $C_4$ depends only on the geometry of $\partial \X$ and on $\alpha$, and the real interval $[1-\epsilon,1]$ can be replaced by any interval from $e^{it}[1-\epsilon,1]\subset \overline{D_\epsilon}$.
So we have
 \begin{equation}\label{mima21}
(|F_z(z)|+|F_{\bar z}(z)|)(1-|z|)^{1-(1+\alpha)\beta}\le C_4,\ \ \ z\in D_\epsilon.
\end{equation}
So from \eqref{fzome}  for $z\in D_\epsilon$ we have
 \begin{equation}\label{mima31}
(|g'(z)|+|h'(z)|)(1-|z|)^{1-(1+\alpha)\beta}\le C_5=C_4+2\|D\omega\|_\infty.
\end{equation}
From Lemma~\ref{hali} and relation \eqref{mima31}, we obtain that $g$ and $h$ are $(1+\alpha)\beta-$H\"older continuous on $\mathbf{D}$. Since $\omega$ is a-priory Lipschitz, it follows that $F$ is $(1+\alpha)\beta-$H\"older continuous on $D_\epsilon$ and so $\phi$ in $[-\epsilon,\epsilon]$.   By repeating the previous procedure starting from the equation \eqref{phib}, but using \begin{equation}\label{phib1}|\phi(t_1) - \phi(t_2)|\le L_2 |t_1-t_2|^{\beta(1+\alpha)}, \ t_1,t_2\in[-\epsilon,\epsilon]\end{equation} instead of \eqref{phib}, we get that $g$ and $h$ are $(1+\alpha)^2\beta-$H\"older continuous on $D_\epsilon$. By using the induction, we get that $g$ and $h$ are $(1+\alpha)^k\beta-$H\"older continuous on $D_\epsilon$. By using one more step, having in mind the relation \eqref{estimate} we get that both $h$ and $g$ are Lipschitz continuous. Since $F^a(z) = g (z) +\overline{h(z)}+\omega(z)$, we obtain that $f$ is Lipschitz continuous near $T_a=\Phi_a(T_\epsilon)$. Since the finite family of  arcs $T_{a_j},$ $j=1,\dots,m$ cover $\partial\X$, it follows that $f$ is Lipschitz near the boundary of $\X$, i.e. in a set $\{z: z\in\X, \mathrm{dist}(\partial \X, z)<\epsilon_1\}$ for a positive constant $\epsilon_1$.
Thus $f$ is Lipschitz on $\X$.
\section{Proof of Theorem~\ref{mainexistq} for $n= 1$}
We recall that we assume without loos of generality that $\X=\A(r,1)$.
First we prove a little weaker result.
\begin{lemma}\label{kisiq}

There exists $r_0>0$ so that $f\in \mathscr{C}^{1,\alpha/2}(\mathbf{D}^+_{r_0})$.
\end{lemma}
Here and in the sequel $\mathbf{D}^+_{r'}=\{z: |z-1|<r'\}\cap\mathbf{D}$.
The proof of Lemma~\ref{kisiq} uses the following lemma.
\begin{lemma}\label{alphagj} Define \begin{equation}\label{squareroot}\sqrt{x}=\left\{
                                                   \begin{array}{ll}
                                                     \sqrt{|x|}, & \hbox{if $x\ge 0$;} \\
                                                     i\sqrt{|x|}, & \hbox{if $x<0$.}
                                                   \end{array}
                                                 \right.\end{equation}
It is clear that $\sqrt{\cdot}$ defined in $\mathbf{R}$ is continuous. Moreover $\Im (\sqrt{x})\ge 0$ for every $x$. Let $P(t)=\sqrt{R(e^{it})}$, where $\sqrt{\cdot}$ is defined in \eqref{squareroot}.
Assume that $R$ is a real $\alpha-$H\"older continuous function in an arc $T\subset\mathbf{T}$. Then $P(z) = \sqrt{R(z)}$ is $\alpha/2-$H\"older continuous function in $T$.
\end{lemma}
\begin{proof}[Proof of Lemma~\ref{alphagj}]
Let  $z,z'\in T$. If $R(z)$ and $R(z')$ have the same sign, then $$|P(z)-P(z')|=|\sqrt{|R(z)|}-\sqrt{|R(z)|}|\le \sqrt{|R(z) - R(z')|}\le \sqrt{C}|z-z'|^{\alpha/2}.$$ If $R(z)$ and $R(z')$ have not the same sign, there exits a point $z''\in T$ between $z$ and $z'$ so that $R(z'')=0$.
Then we get \[\begin{split}|P(z) -P(z')|&=|P(z)-P(z'')+P(z'')-P(z')|\\&\le |P(z)-P(z'')|+|P(z'')-P(z')|\\&\le \sqrt{C}(|z-z''|^{\alpha/2}+|z''-z'|^{\alpha/2})\le 2\sqrt{C}|z-z'|^{\alpha/2}.\end{split}\]
\end{proof}
We also need the following lemma which is the main step of the proof, and whose proof we postpone for the next section.
\begin{lemma}[The key lemma]\label{Ydelta4}
Assume that $u$ and $v$ are two $\mathscr{C}^2$ smooth function in $\mathbf{D}^+_{r_1}$ which are Lipschitz continuous up to the boundary and have bounded laplacian. Assume also that $u(1)=v(1)=0$ and $\gamma$ is $\mathscr{C}^{1,\alpha}$ smooth function in a real interval $[-\epsilon,\epsilon]$. Define   $Y =u(z) - \gamma(v(z))$ and assume that $Y(e^{it})=0$ for $t\in[-r_1,r_1]$. Then there is $0<r_\circ<r_1$ so that $Y\in \mathscr{C}^{1,\alpha}(\overline{{\mathbf{D}^+_{r_\circ}}}).$
\end{lemma}

\begin{proof}[Proof of Lemma~\ref{kisiq}]
Let $b\in\partial \Y$. Since $\partial \Y\in\mathscr{C}^{1,\alpha}$, there is a parameterization $\Gamma(x) = b+(\gamma(x), x):[-\epsilon_1,\epsilon_1]\to \Y$ so that $\Gamma(0)=b$, or there is a parameterization $\Upsilon(x) = b+(x, \nu(x)):[-\epsilon_1,\epsilon_1]\to \Y$ so that $\Upsilon(0)=b$. Moreover, by a small rotation of  the image domain, if needed we can assume that $\gamma'(0)\neq 0$ and $\nu'(0)\neq 0$ so we can assume that  both parameterizations $\Upsilon$ and $\Gamma$ in interval $[-\epsilon_1,\epsilon_1]$ exist.
By making using the translation $w\to w-b$, we can assume that $f(1) = b=0$.
Let $f=u+iv$.
Since $u$ and $v$ are continuous, we can choose $r_1$ so that $u(\mathbf{D}^+_{r_1})\subset (-\epsilon_1,\epsilon_1)$ and $v(\mathbf{D}^+_{r_1})\subset (-\epsilon_1,\epsilon_1)$.
Now
from \eqref{hopf} we conclude that
\begin{equation}\label{difhopf}u_z^2+v_z^2=A:= \frac{\mathbf{c}}{z^2\rho^2(f(z))}.\end{equation}
Further let $Y(z): = u(z) - \gamma(v(z))$. Now recall that $f$ is Lipschitz continuous and in view of \eqref{hgre1}, has bounded Laplacian. Assume that $r_\circ$ is a constant provided by Lemma~\ref{Ydelta4}.
Then
\begin{equation}\label{secequ}u_z= \gamma'(v) v_z +Y_z.\end{equation}

By solving \eqref{difhopf} and \eqref{secequ} we get
\begin{equation}\label{uzu}u_z= \frac{Y_z+\kappa\dot\gamma \sqrt{A(1+\dot\gamma ^2)-Y_z^2}}{1+\dot\gamma ^2},\end{equation}
and
\begin{equation}\label{vzv}v_z=\frac{-\dot\gamma Y_z+ \kappa\sqrt{A(1+\dot\gamma ^2)-Y_z^2}}{1+\dot\gamma ^2}.\end{equation}
Where $\kappa\in\{-1,1\}$. Show that $\kappa=-1$ and show that the above square root function is well-defined continuous function on $T=\partial\mathbf{D}_{r_\circ}^+\cap\mathbf{T}.$
First of all $$\overline{u_z}=\gamma'(v) \overline{v_z} +\overline{Y_z}.$$
Now we have
\begin{equation}\label{jacpoz}J(z,f)= 4\Im(u_z\overline{v_z})=-4\Im (v_z\overline{u_z})\ge 0.\end{equation}
Moreover
$Y(e^{it}) = 0$, $t\in(-r_\circ,r_\circ)$. Then we get for $z=e^{it}\in T$ $$i(zY_z-\overline{zY_z})=0.$$ Hence $$\Im (z Y_z)=0$$ and \begin{equation}\label{realv}(z Y_z)=\Re (z Y_z).\end{equation}
So for $$c_1=\frac{\mathbf{c}}{\rho^2(f(z))}$$ we get
\[\begin{split}v_z\overline{u_z}&=\gamma'(v) |v_z|^2 +v_z\overline{Y_z}
\\&=\gamma'(v) |v_z|^2-\frac{|Y_z|^2\gamma'}{{1+\dot\gamma ^2}} +\kappa\overline{zY_z}\frac{\sqrt{c_1(1+\dot\gamma ^2)-(zY_z)^2}}{|z|^2(1+\dot\gamma ^2)}\\&=\gamma'(v) |v_z|^2-\frac{|Y_z|^2\gamma'}{{1+\dot\gamma ^2}} +\kappa\frac{\sqrt{c_1(1+\dot\gamma ^2)\overline{zY_z}^2-|zY_z|^4}}{|z|^2(1+\dot\gamma ^2)}.
\end{split}\]
Therefore for \begin{equation}\label{rree}R(z) = {c_1(1+\dot\gamma ^2)\overline{zY_z}^2-|zY_z|^4}\end{equation} which is real in $T$ in view of \eqref{realv},  we have
\[\begin{split}\Im (v_z\overline{u_z})&=\Im \left[\kappa\frac{\sqrt{R(z)}}{|z|^2(1+\dot\gamma ^2)}\right]
\\&=\left\{
      \begin{array}{ll}
        \kappa\frac{\sqrt{|R(z)|}}{|z|^2(1+\dot\gamma ^2)}, & \hbox{if $R(z)< 0$;} \\
        0, & \hbox{if $R(z)\ge 0$.}
      \end{array}
   \right.
\end{split}\]
From \eqref{jacpoz} we have $\Im (v_z\overline{u_z})\le 0,$ and this implies that  $\kappa=-1$.

Moreover from Lemma~\ref{Ydelta4} below which is crucial for our approach, we have that $Y\in \mathscr{C}^{1,\alpha}(T)$.

By Lemma~\ref{alphagj} and \eqref{uzu} and \eqref{vzv}, we get that $u$ and $v$ are in $\mathscr{C}^{1,\alpha/2}(T)$. Now Lemma~\ref{heliu} implies that $f\in \mathscr{C}^{1,\alpha/2}(\overline{\mathbf{D}^+_{r_\circ}})$. As $z=1$ is not a special point, there exists the finite family of domains $D_j:=a_j\cdot \overline{\mathbf{D}^+_{r_\circ}}$, $a_j\in \mathbf{T}\subset \partial\X, \ \ j=1,\dots,m$ and a number $r_0>0$  so that $\{x:z\in \X, \ \ r\le 1-r_0\le |z|\le 1\}\subset \cup_{j=1}^m D_j$. It can be taken $m=4$ because $4>\pi$, and $\Phi(T_\epsilon)=a\cdot \{e^{it}, t\in[-1,1]\}$. In order to deal with the inner boundary $r\mathbf{T}$ we make use of the composition $f_1(z)=f(r/z)$ which is a minimizer of $\rho-$energy that maps $\X=\A(r,1)$ onto $\Y$ so that $f_1(\mathbf{T})=f(r\mathbf{T})$ and use the previous case.

This implies that $f\in \mathscr{C}^{1,\alpha/2}(\X)$.
\end{proof}
\begin{proof}[Proof of Theorem~\ref{mainexistq}]
We already proved that $f\in \mathscr{C}^{1,\alpha/2}(\X)$. Let us switch to the mapping $F$ from the proof of Lipschitz continuity and let $F=u+iv$. We know that $F$ is $\mathscr{C}^{1,\alpha}$ if and only if $\mathscr{C}^{1,\alpha}$ near a boundary arc. Denote by abusing the notation $u(t)=u(e^{it})$ and $v(t)=v(e^{it})$. Now recall for $t\in [-r_\circ,r_\circ],$ where $r_\circ>0$ is a constant from Lemma~\ref{Ydelta4}, we have  \begin{equation}\label{vphi9}u(t) =\gamma(v(t)).\end{equation}
We also have \begin{equation}
u'(t)=\gamma'(v(t))v'(t).
\end{equation}
 Recall that we already proved that $u$ and $v$ are Lipschitz continuous. Now we get for $t,s\in [-r_\circ,r_\circ]$, and 
\begin{equation}\label{uno}\begin{split}
|u'(t)-u'(s)|&=|\gamma'(v(t))v'(t)-\gamma'(v(s))v'(s)|\\&\le |\gamma'(v(t))-\gamma'(v(s))|\cdot|v'(t)|+|\gamma'(v(s))|\cdot |v'(t)-v'(s)| \\&\le C|v(t)-v(s)|^\alpha+C|t-s|^{\alpha/2}\cdot |v(s)|^\alpha\\&\le C|t-s|^\alpha+C|t-s|^{\alpha/2}\cdot |s|^{\alpha},\end{split}
\end{equation}
because $v(0)=0$.
Similarly we have

\begin{equation*}\label{due}\begin{split}
|u'(t)-u'(s)|\le C|t-s|^\alpha+C|t-s|^{\alpha/2}\cdot |t|^{\alpha}, \ \ t,s\in [-r_\circ,r_\circ].\end{split}
\end{equation*}
So we get
\begin{equation*}\label{due1}\begin{split}
|u'(t)-u'(s)|\le C|t-s|^\alpha+C|t-s|^{\alpha/2}\cdot \min\{|t|^{\alpha},|s|^{\alpha}\}, \ t,s\in [-r_\circ,r_\circ]\end{split}
\end{equation*}
or what is the same

\begin{equation}\label{due2}\begin{split}
|u'(t)-u'(s)|\le C|t-s|^{\alpha/2}\left(|t-s|^{\alpha-\alpha/2}+\min\{|t|^{\alpha},|s|^{\alpha}\}\right), \ t,s\in [-r_\circ,r_\circ].\end{split}
\end{equation}

So for $t,s\in [-r_\circ,r_\circ]$
\begin{equation}\label{duetre}\begin{split}
|u'(t)-u'(s)|\le C|t-s|^{\alpha/2}\left(|t-s|^{\alpha-\alpha/2}+\min\{|t|^{\alpha-\alpha/2},|s|^{\alpha-\alpha/2}\}\right).\end{split}
\end{equation}

By applying \eqref{rho} for $\mu = \eta=\alpha/2$ we get $$|\partial_t u(e^{it})-\partial_t u(1)|\le C|1-e^{it}|^\alpha, \ t\in [-r_\circ,r_\circ].$$ As $a=1$ is not a special point, and $C$ depends exclusively on the properties of $\partial\Y$,  we get $$|\partial_t u(z)-\partial_t u(z')|\le C|z-z'|^\alpha, \ \ z,z'\in T_{r_\circ/2}=\{e^{it}: t\in(-r_\circ/2,r_\circ/2)\}.$$ By using Lemma~\ref{heliu}, having in mind the relations \eqref{fzome},  we obtain the inequality $$|\partial_t u(z)-\partial_t u(w)|\le C|z-w|^\alpha \ \  z,w\in D_{r_\circ/2},$$ because $\omega\in \mathscr{C}^{1,\alpha}(D_{r_\circ/2})$, where we recall $D_p=\{z:|z|\in(1-p,1]\wedge \mathrm{arg}\,z\in(-p,p)\}$. Since $u =\Re(F)=\Re (g+h+\omega)$ and $$\partial_t u(z) = \Re \left[iz(g'(z) + h'(z))+\Re (i z(\omega_z -\overline{\omega}_z))\right],$$ we infer that \begin{equation}\label{uprim}iz(g'(z) + h'(z))\in\mathscr{C}^\alpha(D_{r_\circ/2}).\end{equation}

By repeating the previous procedure, by interchanging the role of $u$ and $v$, this time by writing the portion  $\partial \Y$ in the form $\Upsilon(x) = (x,\nu(x))$, $x\in(-\epsilon_1,\epsilon_1)$, and by using the new function $X(z)= v(z) - \nu(u(z))$, in view of the formula
$v =\Im(F)=\Im (g+h+\omega)$ we get  $$\partial_t v(z) = \Re \left[z(g'(z) - h'(z))+\Re ( z(\omega_z -\overline{\omega}_z))\right],\  z=re^{it},$$ and thus 
\begin{equation}\label{vprim}z(g'(z) - h'(z))\in\mathscr{C}^{1,\alpha}(D_{r_\circ/2}).\end{equation}
From \eqref{uprim}, \eqref{vprim} and $\omega\in \mathscr{C}^{1,\alpha}(D_{r_\circ/2})$, we get $F$ is in $\mathscr{C}^{1,\alpha}(D_{r_\circ/2})$. As $F=f\circ \Phi_a$ for a certain conformal diffeomorphism $\Phi_a$, it follows that $f$ is $\mathscr{C}^{1,\alpha}$ near a boundary point $a$ of $\X$.
  Thus $f\in \mathscr{C}^{1,\alpha}(\X)$.
\end{proof}

\section{Proof of Theorem~\ref{mainexistq} for $n\ge 2$}
The proof for higher derivatives is now a simple matter. We make use of  Proposition~\ref{2222}. We again assume as in the proof of Lemma~\ref{kisiq} that $b=f(a)\in\partial\Y$ and $\partial \Y\cap \mathbf{D}(b,\delta)$ is a portion of $\partial\Y$ that allows both graphic representations $b+(\gamma(x),x)$ and $b+(x,\nu(x))$, for $x\in (-\epsilon_1,\epsilon_1)$, $\nu(0)=\gamma(0)=0$. We also assume to simplify approach that $b=0$.

We again work with $F=f\circ \Phi$ as in the proof of Lipschitz continuity. Let $F=g+\overline{h}+\omega=u+i v$ and $Y=u-\gamma(v)$, $z\in D_{r_\circ}$.
We first have the equation
\begin{equation}\label{delta}\Delta Y(z) = Q(z)=\Delta u(z) - \gamma''(v(z))|\nabla v(z)|^2 -\gamma'(v)\Delta v.\end{equation}
We already proved that the case $n=1$ and so $Q\in \mathscr{C}^\alpha(D_{r_\circ}\cup T)$, because $\gamma''(v(z))\in \mathscr{C}^\alpha(D_{r_\circ}\cup T)$. Recall that $T = \partial D_{r_\circ}\cap \mathbf{T}$.
 Now  Proposition~\ref{2222} implies that $Y\in \mathscr{C}^{2,\alpha}(D_{r_\circ}\cup T)$. This implies that the case $k=0$ of Proposition~\ref{2222} can be applied, and thus the initial case of the mathematical induction is satisfied. Assume that we have proved it for $n=k-1$. So $Y\in \mathscr{C}^{k+1,\alpha}$.   From \eqref{uzu} and \eqref{vzv} we get that $u$ and $v$ are in $\mathscr{C}^{k+1,\alpha/2}$.
 By abusing the notation, again we have $u(t)=\gamma(v(t))$ for $t\in[-r_\circ, r_\circ]$. Then we get $$u^{(k+1)}(t)=\gamma^{(k+1)}(v(t))v^{k+1}(t)+\gamma'(v(t))v^{(k+1)}(t)+\mathcal{P}$$ where $\mathcal{P}$ is a polynomial expression depending on $\gamma^{(j)}(v(t)), \ j=1,\dots,k$ and $v^{(j)}(t),j=1,\dots, k$.
 In a similar way as in \eqref{uno} we get

\begin{equation}\label{unouno}|u^{(k+1)}(t)-u^{(k+1)}(s)| \le C|t-s|^\alpha+C|t-s|^{\alpha/2}\cdot |s|^{\alpha}, \ \ s,t\in[-r_\circ,r_\circ],\end{equation} where we again use the condition $v(0)=0$.

So for $t,s\in [-r_\circ,r_\circ]$
\begin{equation}\label{duetrek}\begin{split}
|u^{(k+1)}(t)-u^{(k+1)}(s)|\le C|t-s|^{\alpha/2}\left(|t-s|^{\alpha/2}+\min\{|t|^{\alpha/2},|s|^{\alpha/2}\}\right).\end{split}
\end{equation}
By applying \eqref{rho} for $\mu = \eta=\alpha/2$, and remembering that $\omega\in \mathscr{C}^{k+1,\alpha}$ in view of Proposition~\ref{2222}, we get $$|\frac{\partial^{k+1} u(e^{it})}{\partial t^{k+1}}-\frac{\partial^{k+1} u(1)}{\partial t^{k+1}}|\le C|1-e^{it}|^\alpha, \ t\in [-r_\circ,r_\circ].$$ As $a=1$ is not a special point, and $C$ depends exclusively on the properties of $\partial\Y$,  we get $$|\frac{\partial^{k+1} u(z)}{\partial t^{k+1}}-\frac{\partial^{k+1} u(z')}{\partial t^{k+1}}|\le C|z-z'|^\alpha, \ \ z,z'\in T_{r_\circ/2}=\{e^{it}: t\in(-r_\circ/2,r_\circ/2)\}.$$ By using Lemma~\ref{heliu} again, having in mind the relations \eqref{fzome},  we obtain the inequality $$|\frac{\partial^{k+1} u(z)}{\partial t^{k+1}}-\frac{\partial^{k+1} u(z')}{\partial t^{k+1}}|\le C|z-w|^\alpha \ \  z,w\in D_{r_\circ/2},$$ because $\omega\in \mathscr{C}^{k+1,\alpha}(D_{r_\circ/2})$, where we recall $D_p=\{z=re^{it}: r\in[1-p,1)\wedge t\in(-p,p)\}$. As $\partial_t u(z) = \Re (i z(g'+h'))+\Re (\omega_t(z))$, denote $H=g'+h'$. Then \begin{equation}\label{ndimu}\frac{\partial^{k+1} u(z)}{\partial t^{k+1}}  = \Re (i^{k+1} \sum_{j=1}^{k+1} a_j z^j H^{(j)}(z))+\Re \frac{\partial^{k+1} \omega(z)}{\partial t^{k+1}}, \end{equation} where $a_j, \ j=1,\dots, k+1$ are positive integers. Then we get that \begin{equation}\label{Hndim}i^{k+1} \sum_{j=1}^{k+1} a_j z^j H^{(j)}(z)\in \mathscr{C}^{\alpha}(D_{r_\circ/2}).\end{equation}

Now we repeat the previous procedure, by interchanging the role of $u$ and $v$, this time by writing the portion  $\partial \Y$ in the form $\Upsilon(x) = (x,\nu(x))$, $x\in(-\epsilon_1,\epsilon_1)$, and use the new function $X(z)= v(z) - \nu(u(z))$.

As $\partial_t v(z) = \Im ( i z(g'-h'))+\Im (\omega_t(z))$, denote $K=g'-h'$. Then \begin{equation}\label{ndimu2}\frac{\partial^{k+1} v(z)}{\partial t^{k+1}}  = \Im (i^{k+1} \sum_{j=1}^{k+1} b_j z^j K^{(j)}(z))+\Im \frac{\partial^{k+1} \omega(z)}{\partial t^{k+1}}, \end{equation} where $b_j, \ j=1,\dots, k+1$ are positive integers. Then we get that \begin{equation}\label{Kndim}i^{k+1} \sum_{j=1}^{k+1} b_j z^j K^{(j)}(z)\in \mathscr{C}^{\alpha}(D_{r_\circ/2}).\end{equation}

By the mathematical induction we have that $H^{(j)}(z), K^{(j)}(z), \ \ j=1,\dots k$ are smooth in $\overline{D_{r_\circ/2}}$. In view of \eqref{Hndim} and \eqref{Kndim} we therefore have $H^{(k+1)},\ K^{(k+1)}\in \mathscr{C}^{\alpha}(D_{r_\circ/2})$.  Since in addition $\omega\in \mathscr{C}^{k+1,\alpha}(\mathbf{D})$, we get that $F=f\circ \Phi \in  \mathscr{C}^{k+1,\alpha}({D}_{r_\circ/2})$ and consequently $f\in \mathscr{C}^{k+1,\alpha}$ near a boundary point $a\in \partial \X$. The conclusion is that  $f\in \mathscr{C}^{k+1,\alpha}(\X)$. The proof of Theorem~\ref{mainexistq} is completed.

\section{Proof of the key lemma (Lemma~\ref{Ydelta4})}\label{sec4}
\begin{remark}
The proof of Lemma~\ref{Ydelta4} would be much easier if we assume that $\gamma\in \mathscr{C}^2$. In this case Proposition~\ref{trudi} would imply the desired conclusion.
\end{remark}
Let $$G(w,z) = \frac{1}{2\pi}\log \left|\frac{ w-z}{1-w\overline{z}}\right|$$ be the Green function of the unit disk.
We also choose a compactly supported smooth real function $\xi$ so that
\begin{equation}\label{xixi}\xi\in \mathscr{C}^2_0(\mathbf{D}(1,r_1),\mathbf{R}^+), \ \  \xi(z)=1 \ \ \text{for}\ \ \ z\in\overline{\mathbf{ D}}(1,r_1/2).\end{equation} Here $\mathbf{D}(p,\delta):=\{z: |z-p|<\delta\}$.
We will make use the following form of Green theorem
$$\oint_{\gamma} U V_x dy - U V_y dx=\int_{\Omega} (U_x V_x+U_y V_y + U \Delta V) dxdy.$$
Use the notation
$$\mathbf{D}^+_{r}=\{z:|z-1|<r, |z|<1\}.$$
First of all we prove a representation formula for $Y$ and its derivative $Y_{z_i}=\partial_{z_i} Y$, $z=z_1+iz_2$.
\begin{lemma}\label{vule}
Assume that  $Y(z) = u(z) - \gamma(v(z))$ is as in Lemma~\ref{Ydelta4}. Then there are
$\eta, \eta_j\in \mathscr{C}^\infty(\overline{\mathbf{D}^+_{r_1/4}})$, $j=1,2$ so that for $z\in {\mathbf{D}^+_{r_1/4}}$ we have the equation
\begin{equation}\label{aleks}\begin{split}Y(z) &= -\int_{{\mathbf{D}^+_{r_1}}}\gamma'(v(w))\left<\nabla v(w),\nabla_w\left(\xi (w)G(z,w)\right)\right>d\lambda(w)\\& \ \ \ \ -\int_{{\mathbf{D}^+_{r_1}}}\Delta u(w) \xi(w)G(z,w)d\lambda(w)+\eta(z)\end{split}\end{equation}
and
\begin{equation}\label{aleks1}\begin{split}Y(z) &= \int_{{\mathbf{D}^+_{r_1}}}(\gamma'(v(z))-\gamma'(v(w)))\left<\nabla v(w),\nabla_w\left(\xi (w)G(z,w)\right)\right>d\lambda(w)\\& \ \ \ \ +\int_{{\mathbf{D}^+_{r_1}}}(\gamma'(v(z))\Delta v(z)-\Delta u(w)) \xi(w)G(z,w)d\lambda(w)+\eta(z).\end{split}\end{equation}
Moreover for $z=z_1+i z_2$
\begin{equation}\label{kalaj}\begin{split}\partial_{z_j}y(z) &= \int_{{\mathbf{D}^+_{r_1}}}\left[\gamma'(v(z))-\gamma'(v(w))\right]\left<\nabla v(w),\nabla ( \partial_{z_j}G(z,w))\right>d\lambda(w)\\&+
\int_{{\mathbf{D}^+_{r_1}}}\left[\gamma'(v(z))\Delta v(w)-\Delta u(w)\right]\partial_{z_j}G(z,w)d\lambda(w)+\eta_j(z).\end{split}\end{equation}
\end{lemma}
\begin{proof}[Proof of Lemma~\ref{vule}]

As $(\xi(w) -1)G(z,w)$ is well-defined smooth function in $|z|<r_1/4$, because $\xi(w) = 1$ for $|w|\le r_1/2$ and so $(\xi(w) -1)G(z,w)=0$ if $|w|\le r_1/2$ and $|z|<r_1/4$ it is zero on a neighborhood of the diagonal $z=w\in {\mathbf{D}^+_{r_1/4}}$ we can
define
$$\eta(z)=\int_{{\mathbf{D}^+_{r_1}}}Y(w) \Delta((\xi(w) -1)G(z,w))d\lambda(w),$$ which is smooth in $\mathbf{D}^+_{r_1/4}$.
Then by Green identity we have 
$$\eta(z)=-\int_{{\mathbf{D}^+_{r_1}}}\left<\nabla Y(w), \nabla((\xi(w) -1)G(z,w))\right>d\lambda(w)+X,$$
where
\[\begin{split}X&=\int_{\partial {\mathbf{D}^+_{r_1}}} Y(w)\partial_n ((\xi(w) -1)G(z,w))ds(w)
\\&=\int_{\partial {\mathbf{D}^+_{r_1}}} Y(w)(\xi(w) -1)\partial_n (G(z,w))ds(w)
\\&+\int_{\partial {\mathbf{D}^+_{r_1}}} Y(w) G(w,z)\partial_n (\xi(w)-1) ds(w)
\\&=-\int_{\partial {\mathbf{D}^+_{r_1}}} Y(w)\frac{\partial G}{\partial n}ds(w),\end{split}\] because $Y(w)(\xi(w) -1)=-Y(w)$ and $Y(w) G(w,z)=0$ for $w\in \partial\mathbf{D}^+_{r_1}.$

By using the formulas
$$\nabla Y(w) =\nabla u(w) -\gamma'(v(w)) \cdot \nabla v(w),$$
$$\nabla((\xi(w) -1)G(z,w))=\nabla((\xi(w)G(z,w))-\nabla(G(z,w))$$
and then again  Green identity we get
$$\int_{{\mathbf{D}^+_{r_1}}}\left<\nabla u(w), \nabla((\xi(w))G(z,w))\right>d\lambda(w)=-\int_{{\mathbf{D}^+_{r_1}}}\Delta u(w) \xi(w) G(z,w)d\lambda(w).$$
So again by Green identity we have
\[\begin{split}\eta&-\int_{{\mathbf{D}^+_{r_1}}}\Delta u(w) \xi(w) G(z,w)d\lambda(w)-\int_{{\mathbf{D}^+_{r_1}}}\gamma'\left<\nabla v, \nabla (\xi G(z,w))\right> d\lambda(w) \\&= \int_{{\mathbf{D}^+_{r_1}}}\left<\nabla Y(w), \nabla ( G(z,w))\right>d\lambda(w)-\int_{\partial {\mathbf{D}^+_{r_1}}} Y(w)\frac{\partial G(w,z)}{\partial n}ds(w)
\\&=-\int_{{\mathbf{D}^+_{r_1}}}Y(w) \Delta(G(z,w))d\lambda(w)=Y(z), \  \ z\in \mathbf{D}^+_{r_1}.\end{split}\]
This proves \eqref{aleks}.
Now by using one more time Green formula we get $$\int_{{\mathbf{D}^+_{r_1}}}\left<\nabla u,\nabla(\xi(w) G(z,w))\right>d\lambda(z)=-\int_{{\mathbf{D}^+_{r_1}}}\Delta  u(w)\xi(w) G(z,w)d\lambda(z)$$ because $$\int_{\partial \mathbf{D}^+_{r_1}}\xi(w) G(z,w)\frac{\partial u}{\partial n} ds(w)=0$$ where we used the relation $ \xi(w) =0$ for $w\in \partial \mathbf{D}^+_{r_1}\setminus \mathbf{T}$ and $G(z,w)=0$ for $w\in \mathbf{T}$.

Moreover
\[\begin{split}\partial_{z_j}Y(z) &= \int_{{\mathbf{D}^+_{r_1}}}\left[\gamma'(v(z))-\gamma'(v(w))\right]\left<\nabla v(w),\nabla (\xi \partial_{z_j}G(z,w))\right>d\lambda(w)\\&+
\int_{{\mathbf{D}^+_{r_1}}}\left[\gamma'(v(z))\Delta v(w)-\Delta u(w)\right]\xi(w)\partial_{z_j}G(z,w)d\lambda(w)+\partial_{z_j}\eta(z)\end{split}\]
where
\[\begin{split}\eta_j(z)&= \int_{{\mathbf{D}^+_{r_1}}}\left[\gamma'(v(z))-\gamma'(v(w))\right]\left<\nabla v(w),\nabla ((\xi-1) \partial_{z_j}G(z,w))\right>d\lambda(w)\\&+
\int_{{\mathbf{D}^+_{r_1}}}\left[\gamma'(v(z))\Delta v(w)-\Delta u(w)\right](\xi(w)-1)\partial_{z_j}G(z,w)d\lambda(w)+\partial_{z_j}\eta(z).\end{split}\]
This proves \eqref{kalaj} in view of calculations which we derive in the following proof which among the other facts confirms that the differentiation is possible inside the integral.
\end{proof}

\begin{proof}[Proof of Lemma~\ref{Ydelta4}]
By Lemma~\ref{vule}, for $z=z_1+i z_2$ we have
\[\begin{split}\partial_{z_j}Y(z) &=\Theta(z)+\Lambda(z) +\eta_j(z), \end{split}\] where
$$\Theta(z)=\int_{{\mathbf{D}^+_{r_2}}}\left[\gamma'(v(z))-\gamma'(v(w))\right]\left<\nabla v(w),\nabla ( \partial_{z_j}G(z,w))\right>d\lambda(w),$$
$$\Lambda(z)=\int_{{\mathbf{D}^+_{r_2}}}(\gamma'(v(z))\Delta v(w)-\Delta u(w))\partial_{z_j}G(z,w)d\lambda(w),$$
and
$$\eta_j(z)\in \mathscr{C}^\infty(\overline{\mathbf{D}^+_{r_3}}),\ \ r_3<r_2, j=1,2.$$
Let $$ \ m(z)=\gamma'(v(z)),\ \  \ \ h(z,w):= \left<\nabla v(w),\nabla ( \partial_{z_j}G(z,w))\right>.$$
Then
$$\Theta(z)=\int_{{\mathbf{D}^+_{r_2}}}(m(z) - m(w))h(z,w)d\lambda(w).$$
Let $r_4<r_3/2$. Now for $z,z'\in \mathbf{D}^+_{r_4}$,
\[\begin{split}\Theta(z') - \Theta(z)&=\int_{\mathbf{D}^+_{r_2}}(m(z') - m(w))h(z',w)d\lambda(w)\\&-\int_{{\mathbf{D}^+_{r_2}}}(m(z) - m(w))h(z,w)d\lambda(w)
\\&=J_1+J_2+J_3+J_4,\end{split}\] where for  $\zeta=(z+z')/2$, $\sigma=|z-z'|$ and $G=\mathbf{D}(\zeta,\sigma)\cap \mathbf{D}$
\[
\begin{split}
J_1&=\int_G h(z',w) (m(z')-m(w))d\lambda(w)
\\J_2&=\int_G h(z,w) (m(w)-m(z))d\lambda(w)
\\J_3 &=\int_{\mathbf{D}^+_{r_2}\setminus G} h(z',w) (m(z')-m(z))d\lambda(w)
\\J_4 &=\int_{\mathbf{D}^+_{r_2}\setminus G} (h(z',w)-h(z,w)) (m(z)-m(w))d\lambda(w).\end{split}\]
For $J_1$, in view of boundedness of $|\nabla v|$  we get the inequality $$|\left<\nabla v(w),\nabla ( \partial_{z_j}G(z',w))\right>|\le \frac{C}{|z'-w|^2}.$$ 
The constant $C$ that appear in the proof is not the same and its value can vary from one to the another appearance. Because of  H\"older continuity of $\gamma'$ we therefore get
\[\begin{split}|J_1|&\le C \int_{G}|z'-w|^{2-\alpha}d\lambda(w)\\&\le C \int_{|w-z'|<3/2\sigma }|w-z'|^{2-\alpha}d\lambda(w)\\&=\frac{2\pi}{\alpha}(3\sigma/2)^\alpha=C|z-z'|^\alpha .\end{split}\]
Similarly we obtain $$|J_2|\le C|z-z'|^\alpha.$$
To estimate $J_3$ we first recall that $z,z'\in \mathbf{D}_{r_4}^+$, and so $\partial G \cap \partial \mathbf{D}_{r_2}=\emptyset$. Let $T'=G\cap \mathbf{T}$ and $T''= \partial \mathbf{D}_{r_2}^+\setminus T'$.

Then by using the Green formula we get
\[\begin{split}\int_{\mathbf{D}^+_{r_2}\setminus G} h(z',w) d\lambda(w)&=\int_{\mathbf{D}^+_{r_2}\setminus G} \left<\nabla v(w),\nabla ( \partial_{z_j}G(z,w))\right>d\lambda(w)
\\&=-\int_{\mathbf{D}^+_{r_2}\setminus G} \Delta v(w)  \partial_{z_j}G(z,w)d\lambda(w)\\&\ \ \ \ \ +
\int_{\partial(\mathbf{D}^+_{r_2}\setminus G)}   \partial_{z_j}G(z,w))\partial_n v ds(w)
\end{split}\]
Further \[\begin{split}
\int_{\mathbf{D}^+_{r_2}\setminus G} |\Delta v(w)  \partial_{z_j}G(z,w)|d\lambda(w)&\le C \int_{\mathbf{D}}\frac{d\lambda(w)}{|z-w|}\le C \int_{\mathbf{D}}\frac{d\lambda(w)}{|w|}=C\pi=C.\end{split}\]
Next,
\[\begin{split}|\int_{\partial(\mathbf{D}^+_{r_2}\setminus G)}   \partial_{z_j}G(z,w)\partial_n v ds(w)|&\le |\int_{\mathbf{D}\cap
 \partial(\mathbf{D}^+_{r_2})}   \partial_{z_j}G(z,w)\partial_n v ds(w)|\\&+|\int_{T''}   \partial_{z_j}G(z,w)\partial_n v ds(w)|
\\&+|\int_{\mathbf{D}\cap \partial G}   \partial_{z_j}G(z,w)\partial_n v ds(w)|\\&=I_1+I_2+I_3.
\end{split}\]
Further $$I_1=|\int_{\mathbf{D}\cap \partial(\mathbf{D}^+_{r_2})}   \partial_{z_j}G(z,w)\partial_n v ds(w)|\le C\int_{|w-1|=r_2}\frac{1}{|w-z|}ds(w).$$

Now recall that $z\in \mathbf{D}^+_{r_4}$, where $r_4<r_2/2$. Therefore $|w-z|\ge |w-1|-|z-1|>r_2/2$. Hence

$$I_1\le C \frac{2}{r_2}\cdot 2\pi r_2= C.$$
If $w\in T''(\subset\mathbf{T})$ then we have $|\partial_{z_j}G(z,w)|=0$, and so
$$I_2=\int_{T''}   \partial_{z_j}G(z,w)\partial_n v ds(w)=0.$$
Further
\[\begin{split}
|I_3|&=|\int_{\mathbf{D}\cap \partial G}   \partial_{z_j}G(z,w)\partial_n v ds(w)|\\&\le C \int_{|w-\zeta|=|z-z'|}\frac{d\lambda(w)}{|w-z|}.\end{split}\]
By using the inequalities $|w-z|\ge |w-\frac{z+z'}{2}|-\frac{|z-z'|}{2}\ge |z-z'|-\frac{|z-z'|}{2}$, we get
$$|\int_{\mathbf{D}\cap \partial G}   \partial_{z_j}G(z,w)\partial_n v ds(w)|\le C\frac{2}{|z-z'|}2\pi |z-z'|=C.$$
So $$|J_3|\le C|z-z'|^\alpha.$$
Now we deal with  $J_4$. We have $$|J_4|\le \int_{\mathbf{D}^+_{r_2}\setminus G} |h(z',w)-h(z,w)| |m(z)-m(w)|d\lambda(w).$$
Now \[\begin{split}|h(z',w)-h(z,w)|&=|\left<\nabla v(w),\nabla ( \partial_{z_j}G(z,w))-\nabla ( \partial_{z_j}G(z',w))\right>|\\&\le C|\nabla ( \partial_{z_j}G(z,w))-\nabla ( \partial_{z_j}G(z',w))|\\&= C|z-z'||\nabla ( \partial_{z_j}\partial_{z_j}G(\hat{z},w))|
\\&\le C|z-z'|\cdot |\hat{z}-w|^{-3} \end{split}\] for some $\hat z\in [z,z']$.
Thus for $\zeta= (z+z')/2$, $$J_4\le C|z-z'|\int_{|w-\zeta|\ge \sigma}\frac{|z-w|^\alpha}{|\hat{z}-w|^3}d\lambda(w).$$
For  $|w-\zeta|\ge \sigma=|z-z'|$ we get 

$$|z-w|\le |w-\zeta|+|\zeta-z|\le |w-\zeta|+|\frac{z-z'}{2}|\le |w-\zeta|+\frac{1}{2}|w-\zeta|=\frac{3}{2}|w-\zeta|$$ and
$$|w-\zeta|\le |w-\hat{z}|+|\hat{z}-\zeta|\le |w-\hat{z}|+\frac{\sigma}{2}\le |w-\hat{z}|+(|w-\zeta|-|\hat{z}-\zeta|)\le 2|w-\hat{z}|.$$
Therefore 
\begin{equation}\label{ppll}|z-w|\le 3/2 |w-\zeta|\le 3|\hat{z}-w|.\end{equation} So $$J_4 \le C|z-z'|^\alpha.$$
Then for $\tilde{z}=\frac{z-z'}{2}$, $\zeta=(z+z')/2$, $p=|\tilde{z}|$   by using the simple formula $$(w-z)(w-z')= (w-\zeta-\tilde{z})(w-\zeta+\tilde{z})$$ we get \[\begin{split}\int_{|w|<1}&\frac{1}{|w-z||w-z'|}d\lambda(w) \le \int_{|w|<2}\frac{1}{|w^2-\tilde{z}^2|}d\lambda(w)
\\&= \int_0^{2\pi}\int_0^2 \frac{r}{\sqrt{r^4+p^4-2 r^2 p^2 \cos(2t)}}dr dt
\\&=\frac{1}{2}\int_0^{2\pi} \log\left[4 - p^2 \cos(2 t) + \sqrt{16 + p^4 - 8 p^2 \cos(2 t)}\right] -   \log[2 p^2 \sin^2 t]dt
\\&\le \pi \log[4 +1 + \sqrt{25}]+\log 1/p-\pi \log 2=\pi\log \frac{5}{p}.\end{split}\]

Because $\Delta u$, $\Delta v$, $|\nabla u|$ are bounded by a constant and $\gamma'$ is $\alpha-$H\"older continuous we get that
\[\begin{split}|\Lambda(z) - \Lambda(z')|&\le C|z-z'|^\alpha \int_{\mathbf{D}^+_{r_2}}|\partial_{z_j}G(z,w)|d\lambda(w)\\& + C \int_{\mathbf{D}^+_{r_2}}|\partial_{z_j}G(z,w))-\partial_{z_j}G(z,w))|d\lambda(w)
\\&\le C|z-z'|^\alpha + C \int_{\mathbf{D}^+_{r_2}}\frac{|z-z'|}{|w-z'| |w-z|}d\lambda(w)\\&\le C|z-z'|^\alpha +C |z-z'|\pi\log \frac{10}{|z-z'|}
\\ & \le C|z-z'|^\alpha.
\end{split}\]
Combining the above estimates, remembering that $\eta_j(z)$ is a smooth function in $\overline{\mathbf{D}_{r_4}^+}$, we conclude that
there is a constant $C$ so that for $z,z'\in \overline{\mathbf{D}_{r_4}^+}$ we have $$|\partial_{z_j}Y (z)-\partial_{z_j}Y (z')|\le C|z-z'|^\alpha, j=1,2, \ \ z=z_1+iz_2$$ and this concludes the proof of the key lemma.

\end{proof}

\section{Concluding remark}
In Remark~\ref{vanesa} has been verified that, for Euclidean setting, there exists a minimizing diffeomorphism
whose inverse is not $\mathscr{C}^{1,\alpha}$ up to the boundary, so Theorem~\ref{mainexistq} cannot be improved.
In the following remark we explain that for radial metrics always exists a minimizing diffeomorphism, whose inverse is not $\mathscr{C}^{1,\alpha}$ up to the boundary.
\begin{remark}\label{rema}
Assume that $\varrho(s), R \le s\le 1 $ is a smooth non-vanishing function and let $\rho(w)=1/\varrho(|z|)$. Assume further that $t\rho(t)$ is monotonous.
In \cite{london} are  found all examples  $w$ of radial
$\rho$-harmonic maps between annuli and all they minimizes the $\rho-$energy provided the Nitsche type condition \eqref{nit} is satisfied. The mapping $w$, up to the rotation of annuli is given by
$w(se^{it}) =
q^{-1}(s)e^{it},$ where
\begin{equation}\label{var}q(s) =\exp\left(\int_{1}^s
\frac{dy}{\sqrt{y^2+\gamma\varrho^2 }}\right), \ R\le s\le
1,
\end{equation} and $\gamma$ satisfies the condition:
\begin{equation}\label{unt} y^2+\gamma\varrho^2(y)\ge 0, \;\text{for} \; R\le y\le 1.\end{equation}
The mapping $w$ is a $\rho$-harmonic mapping between annuli
$\A=\A(r,1)$ and $\A'=\A(R,1)$, where
\begin{equation}\label{rr}r=\exp\left(\int_{1}^R \frac{dy}{\sqrt{y^2+\gamma\varrho^2
}}\right).\end{equation} The harmonic mapping $w$ is normalized by
$w( e^{it})=e^{it}.$ The mapping $w=h^\gamma(z)$ is a diffeomorphism up to the boundary, and is called  \emph{$\rho$-Nitsche map}.

For \begin{equation}\label{diamond}\gamma_\diamond=-\min_{R\le y\le 1}y^2\rho^2(y)=-\min\{R^2\rho^2(R),\rho^2(1)\},\end{equation}
we have well defined function
\begin{equation}\label{qu}q_{\diamond}(s)=\exp\left(\int_{1}^s
\frac{dy}{\sqrt{y^2+\gamma_\diamond\varrho^2(y) }}\right),
R\le s\le 1.\end{equation} The mapping $h_{\diamond}:\A\to
\A'$ defined by $h_{\diamond}(se^{it})=q_{\diamond}^{-1}(s)e^{it}$ is
called the \emph{critical Nitsche map}.


 If $r<1$, then in \cite{london} is proved that there exists a radial $\rho-$ harmonic mapping of the annulus
$\A=\A(r,1)$ onto the annulus $\A'=\A(R,1)$ if and only if
\begin{equation}\label{nit}r\ge r_\diamond:=\exp\left(\int_{1}^{R}
\frac{dy}{\sqrt{y^2+\gamma_\diamond\varrho^2(y)
}}\right).\end{equation}
It is clear that
\begin{equation}\label{diadia}
r_\diamond <{R}.
\end{equation}

For every Nitsche map $w=h^\gamma(z)=p(s)e^{it}$, where $z=s e^{it}$
and $q(s)=p^{-1}(s)$ we have
$\mathrm{Hopf}(w)=\frac{\gamma}{4 z^2}.$

Since $$q'_{\diamond}(s)=\exp\left(\int_{1}^s
\frac{dy}{\sqrt{y^2+\gamma_\diamond\varrho^2(y) }}\right) \frac{1}{\sqrt{s^2+\gamma_\diamond\varrho^2(s)}}, $$ we get $q'_{\diamond}(1)=\infty$ or $q'_{\diamond}(r)=\infty$, because of \eqref{diamond}. Thus $h_\diamond^{-1}(se^{it})=q_\diamond(s)e^{it}$ is not smooth up to the boundary.

\end{remark} 

Now Remark~\ref{rema}, Theorem~\ref{mainexistq}  and Proposition~\ref{q4} lead to the following conjecture
\begin{conjecture}
Assume that $\rho$ is a metric in $\A(R,1)$ with bounded Gaussian curvature and finite area. Define  $r_\diamond$ as the infimum of all $r$ so that there exists a minimizing $\rho-$harmonic diffeomorphism between $\A(r,1)$ and $\A(R,1)$. We know from Proposition~\ref{q4} that $r_\diamond\le R$. Then we conjecture that  $r_\diamond<R$ and for $r>r_\diamond$ there exists a $\rho-$minimizing diffeomorphism between $\A(r,1)$ and $\A(R,1)$ which is is $\mathscr{C}^{1,\alpha}$ up to the boundary together with its inverse. If $\Y$ is a double connected bounded
domain in the complex plane, there exist a conformal diffeomorphism $\Psi:\A(R,1)\onto \Y$, where $R=R(\Y)\in(0,1)$. In this case $\rho=|\Psi'(z)|$ is a smooth metric with a bounded Gaussian curvature in $\Y$ and finite area. Namely, $K\equiv 0$ and $A(\rho)=\int_{\A(R,1)}|\Psi'(z)|^2d\lambda(z) = \mathrm{Area}(\Y)$. This in turn implies that a special case of the previous conjecture is the following conjecture.
There exists $r_\diamond<R(\Y)$, so that for $r>r_\diamond$ there exists a Euclidean harmonic diffeomorphism $h:\A(r,1)\onto \Y$ that minimizes the energy and it is $\mathscr{C}^{1,\alpha}$ together with its inverse up to the boundary  if $\partial \Y \in \mathscr{C}^{1,\alpha}$.
\end{conjecture}


\begin{thebibliography}{1}

\bibitem{AIM}
\textsc{K. Astala, T. Iwaniec, and G. J. Martin,}  \textit{Deformations of
annuli with smallest mean distortion}. Arch. Ration. Mech. Anal.
{\bf 195}, 899--921 (2010).
\bibitem{cris}
\textsc{J. Cristina, T. Iwaniec, L. V.  Kovalev, J. Onninen,}
\emph{The Hopf-Laplace equation: harmonicity and regularity.}
Ann. Sc. Norm. Super. Pisa, Cl. Sci. (5) \textbf{13}, No. 4, 1145--1187 (2014).
\bibitem{dht}
\textsc{U. Dierkes, S. Hildebrandt, A. J. Tromba},
\emph{Regularity of minimal surfaces,}
 Grundlehren der mathematischen Wissenschaften 340
 Springer-Verlag Berlin Heidelberg, 2010.



    \bibitem{halmos}
\textsc{P.~R. Halmos,} \emph{Measure {T}heory}, D. Van Nostrand Company,
Inc., New York,
  N. Y., 1950.


\bibitem{EH} \textsc{E. Heinz:} {\it On certain nonlinear elliptic
differential equations and univalent mappings,} J. d' Anal. \textbf{5},
 197-272 (1956/57).

\bibitem{gt} \textsc{ D. Gilbarg; N. Trudinger:} \emph{Elliptic Partial
Differential Equations of Second Order}, Vol. 224, 2 Edition,
Springer 1977, 1983.
\bibitem{G}
\textsc{G. M. Goluzin:} {\it Geometric function theory of a Complex Variable}, --Transl. Of Math. Monographs 26. - Providence: AMS, 1969.
\bibitem{IKO2}
\textsc{T. Iwaniec, L. V. Kovalev and J. Onninen:} \textit{The Nitsche
conjecture}, J. Amer. Math. Soc. \textbf{24}, 345--373 (2011).
\bibitem{iwa}
\textsc{T. Iwaniec, K.-T. Koh, L. Kovalev, J. Onninen}, \emph{Existence of energy-minimal diffeomorphisms between doubly connected domains.}
Invent. Math. {\bf 186}, No. 3, 667--707 (2011).

\bibitem{Job1} \textsc{J. Jost:}
\textit{Minimal surfaces and Teichm\"uller theory}.  Yau, Shing-Tung
(ed.), Tsing Hua lectures on geometry and analysis, Taiwan, 1990--91.
Cambridge, MA: International Press. 149--211 (1997).

\bibitem{jost}
\textsc{J. Jost:} \emph{Two-dimensional geometric variational problems}, John Wiley and  Sons, Ltd.,
Chichester, 1991.

\bibitem{jost3}
\textsc{J. Jost:} \emph{
Harmonic maps between surfaces (with a special chapter on conformal mappings).}
Lecture Notes in Mathematics. 1062. Berlin etc.: Springer-Verlag. X, 133 p.;  (1984).
  \bibitem{aim}
\textsc{ D. Kalaj}, \emph{Cauchy transform and Poisson's equation,} Advances in Mathematics,  \textbf{231}, No. 1,  213-242 (2012).
\bibitem{london}
\textsc{D. Kalaj,}
\emph{Deformations of annuli on Riemann surfaces and the generalization of Nitsche conjecture.}
J. Lond. Math. Soc., II. Ser. 93, No. 3, 683--702 (2016).
\bibitem{kal0}
\textsc{D. Kalaj},
\emph{Energy-minimal diffeomorphisms between doubly connected Riemann surfaces. }
Calc. Var. Partial Differ. Equ. \textbf{51}, No. 1-2, 465--494 (2014).
\bibitem{kal}
\textsc{D. Kalaj,}
\emph{Lipschitz property of minimisers between double connected curfaces.}
J Geom Anal \textbf{30}, 4150--4165 (2020) https://doi.org/10.1007/s12220-019-00235-x.
\bibitem{kalmat}
\textsc{D. Kalaj, M. Mateljevic:}
\emph{$(K,K')$-quasiconformal harmonic mappings.}
Potential Anal. \textbf{36}, No. 1, 117--135 (2012).

\bibitem{Ka}
\textsc{D. Kalaj: } {\em On the Nitsche conjecture for harmonic mappings in $R^2$ and $R^3$.} Israel J. Math. {\bf 150}, 241--251 (2005).
\bibitem{kalam}
\textsc{D. Kalaj,
B. Lamel:}
\emph{Minimisers and Kellogg's theorem.}
Math. Ann. \textbf{377}, No. 3-4, 1643--1672 (2020).

\bibitem{Ko} \textsc{O. Kellogg:} {\it Harmonic functions and Green's
integral,} Trans. Amer. Math. soc. vol. \textbf{13} (1912) pp. 109--132.
\bibitem{Kid}
\textsc{D. Kinderlehrer:}
\emph{The boundary regularity of minimal surfaces.}
Ann. Sc. Norm. Super. Pisa, Sci. Fis. Mat., III. Ser. \textbf{23}, 711--744 (1969).

\bibitem{Les0}
\textsc{F. D. Lesley} {\it
Differentiability of minimal surfaces at the boundary.}
Pac. J. Math. \textbf{37}, 123--139 (1971).
\bibitem{Lw}
\textsc{F. D. Lesley, S. E. Warschawski:} {\it Boundary behavior of the
Riemann mapping function of asymptotically conformal curves} Math.
Z. \textbf{179} (1982), 299--323.


\bibitem{les} \textsc{F. D. Lesley, S. E. Warschawski,} \emph{On Conformal Mappings with Derivative in VMOA,} Math. Z. \textbf{158}, 275--283 (1978).
\bibitem{L}
\textsc{A. Lyzzaik:} {\em The modulus of the image annuli under univalent harmonic mappings and a conjecture of J.C.C. Nitsche},  J. London Math. Soc., {\bf 64}, 369--384 (2001).

\bibitem{markovic}
\textsc{V. Markovi\'c}:
\emph{Harmonic Maps and the Schoen Conjecture}, J. Am. Math. Soc.
Journal Profile
 \textbf{30}, 799--817 (2017).
\bibitem{nit}
\textsc{J. C. C. Nitsche:}
\emph{The boundary behavior of minimal surfaces. Kellogg 's theorem and branch points on the boundary.}
Invent. Math. \textbf{8}, 313--333 (1969).

\bibitem{Nitsche}
\textsc{J.C.C. Nitsche:} {\em On the modulus of doubly connected regions under harmonic mappings},  Amer. Math. Monthly,  {\bf 69}, 781--782 (1962).


 \bibitem{MP} \textsc{M. Pavlovi\' c:}
{\it Boundary correspondence under harmonic quasiconformal
homeomorfisms of the unit disc}, Ann. Acad. Sci. Fenn., \textbf{27}, 365--372 (2002).
\bibitem{chp}
\textsc{Ch. Pommerenke:}
\emph{Boundary behaviour of conformal maps.}
Grundlehren der Mathematischen Wissenschaften. 299. Berlin: Springer- Verlag. ix, 300 p. (1992).
\bibitem{tam}
\textsc{L. Tam and T. Wan:} {\it Quasiconformal harmonic diffeomorphism and
universal Teichm\" uler space}, J. Diff. Geom. {\bf 42} (1995)
368-410.

\bibitem{simon}
\textsc{L. Simon}
\emph{A H\"older estimate for quasiconformal maps between surfaces in Euclidean space}.
Acta Math.
\textbf{139}, 19--51 (1977).

\bibitem{SW}
\textsc{S. E. Warschawski:} {\it On Differentiability at the
Boundary in Conformal Mapping,} Proc. Amer. Math. Soc., 12:4 (1961),
614-620.

\bibitem{w1}\textsc{S.~E. Warschawski:} {\it On differentiability at the boundary in
conformal mapping,} Proc. Amer. Math. Soc, {\bf 12} (1961), 614--620.
\bibitem{w2} \bysame: {\it On the higher derivatives at the boundary in conformal
mapping,} Trans. Amer. Math. Soc. \textbf{38}, No. 2 (1935),
310--340.
\bibitem{sw3}
\textsc{S.~E. Warschawski:} {\it
\"Uber das Randverhalten der Ableitung der Abbildungsfunktion bei konformer Abbildung.}
Math. Z. \textbf{35}, 321--456 (1932).


 \bibitem{W} \textsc{A. Weitsman:} {\em Univalent harmonic mappings of annuli and a conjecture of J.C.C. Nitsche}, Israel J. Math., {\bf 124},  (2001),  327--331.
\end{thebibliography}
\end{document}